\documentclass[leqno]{article}

\usepackage{amsmath,amssymb,amsthm,bm,mathrsfs}
\newcommand{\R}{\mathbb{R}}

\newcommand{\PP}{\mathbb{P}}

\newcommand{\Div}{\mathrm{div}\,}

\newcommand{\pt}{\partial}
\newcommand{\ep}{\varepsilon}

\newtheorem{definition}{Definition}[section]
\newtheorem{proposition}{Proposition}[section]
\newtheorem{theorem}{Theorem}[section]
\newtheorem{remark}{Remark}[section]
\newtheorem{lemma}{Lemma}[section]

\newtheorem*{assumption}{Assumption}
\makeatletter

\@addtoreset{equation}{section}
\makeatother

\usepackage{layout}
\title{Remark on the strong solvability of the Navier-Stokes equations in the weak $L^n$ space}
\author{Takahiro Okabe \\
{\normalsize Depertment of Mathematics Education,}\\
{\normalsize Hirosaki University,}\protect\\
{\normalsize Hirosaki 036-8560, Japan}\protect\\
{\normalsize E-mail:\texttt{okabe@hirosaki-u.ac.jp}}
\and
Yohei Tsutsui \\
{\normalsize Department of Mathematical Sciences,} \\
{\normalsize Shinshu University, }\\
{\normalsize Matsumoto, 390-8621, Japan}\\
{\normalsize E-mail:\texttt{tsutsui@shinshu-u.ac.jp}}
}

\date{}
\begin{document}
\maketitle
\begin{abstract}
The initial value problem of the incompressible 
Navier-Stokes equations with non-zero forces in $L^{n,\infty}(\R^n)$ is investigated. 
Even though the Stokes semigroup 
is not strongly continuous on $L^{n,\infty}(\R^n)$,
with the qualitative condition 
for the external forces,
it is clarified that the mild solution of the
Naiver-Stokes equations satisfies the
differential equations 
in the topology of $L^{n,\infty}(\R^n)$. 
Inspired by the conditions for the forces,
we characterize the maximal complete subspace 
in $L^{n,\infty}(\R^n)$
where the Stokes semigroup is strongly continuous
at $t=0$. 
By virtue of this subspace, we also show local
well-posedness of the strong solvability of the 
Cauchy problem
without any smallness condition on the initial data in the subspace. 
Furthermore, via existence of local solutions,
we extend the uniqueness criterion 
in the solution class $BC\bigl([0,T)\,;\,L^{n,\infty}(\R^n)\bigr)$ for wider class of initial data, compared with the above subspace. \end{abstract}
\textbf{Key word:} Navier-Stokes equations, Strong solution, Lorentz spaces
\\
\textbf{MSC(2010):} 35Q30; 76D05
%
\section{Introduction}
%
Let $n\geq 3$. We consider 
the incompressible Naiver-Stokes equations in the whole space $\R^n$:
\begin{equation*}\tag{N-S}
\left\{
\begin{split}
&\pt_t u -\Delta u + u\cdot \nabla u + \nabla \pi=f
\qquad \text{in } \R^n\times (0,T),\\
& \Div u=0
\qquad \text{in } \R^n\times (0,T),\\
&u(\cdot,0)=a 
\qquad \text{in } \R^n.
\end{split}\right.
\end{equation*}
Here $u=u(x,t)=\bigl(u_1(x,t),\dots,u_n(x,t)\bigr)$
and $\pi=\pi(x,t)$ are the unknown velocity and the
pressure of the incompressible fluid at $(x,t)\in \R^n\times (0,T)$, respectively.
While, $a=a(x)=\bigl(a_1(x),\dots,a_n(x)\bigr)$
and
$f=f(x,t)=\bigl(f_1(x,t),\dots,f_n(x,t)\bigr)$
are the given initial data and external force, respectively.

In this paper, 
we study the strong solvability of 
the Naiver-Stokes equations 
in the framework of the weak Lebesgue space 
$L^{n,\infty}(\R^n)$
with non-zero external forces.
In particular, 
introducing the maximal subspace
in $L^{n,\infty}(\R^n)$ where the Stokes operator
is strongly continuous,
we  consider the local in time well-posedness
of the initial value problem 
of (N-S) in the subspace.
Due to the existence o the local solution,
we discuss the uniqueness in 
$BC\bigl([0,T)\,;\,L^{n,\infty}(\R^n)\bigr)$ 
of weak mild solutions of (N-S).

The strong solvability of the (N-S) in 
the Lebesgue and the Sobolev spaces,
in terms of the semigroup theory, 
was developed by Fujita and Kato 
\cite{Fujita Kato ARMA1964}, 
Kato \cite{Kato MZ1984} and Giga and Miyakawa
\cite{Giga Miyakawa ARMA1985}, and so on.
However, it is well-known
that the weak Lebesgue space $L^{n,\infty}(\R^n)$ has lack of the density of compact-supported functions $C_{0}^\infty(\R^n)$ and that
the Stokes operator $\{e^{t\Delta}\}_{t\geq 0}$
is not strongly continuous at $t=0$ in $L^{n,\infty}(\R^n)$.
Therefore, there are difficulty for the validity of the differential equation:
\begin{equation*}
\frac{d}{dt}u -\PP \Delta u +
\PP [u\cdot\nabla u]=\PP f, \qquad t>0
\end{equation*}
in the critical topology of $L^{n,\infty}(\R^n)$,
especially, with non-trivial external forces
and for the verification of the
the local in time existence 
and also the uniqueness 
of mild solutions of (N-S) 
for initial data in $L^{n,\infty}(\R^n)$,
where $\PP$ denotes the Leray-Hopf, 
the Weyl-Helmholtz or the Fujita-Kato 
bounded projection. 
For the Cauchy problem, in case $f\equiv 0$, Miyakawa and Yamada \cite{Miyakawa Yamada HMJ1992} constructed the mild solution $u\in C\bigl((0,\infty)\,;\,L^{2,\infty}(\R^2)\bigl)$
with $u(t)\rightharpoonup a$ weakly $*$ in $L^{2,\infty}(\R^2)$. 
Barraza \cite{Barraza RMI1996}
proved the existence of a global mild solution
$u \in BC\bigl((0,\infty)\,;\,L^{n,\infty}(\R^n)\bigr)$
with small initial data.
As for local in time solution, 
Kozono and Yamazaki \cite{Kozono Yamazaki HJM1995}
constructed a regular solution $u(t)$ 
in the framework of 
$L^{n,\infty}(\Omega)+L^r(\Omega)$ where $r>n$ and
$\Omega$ is a exterior domain.
By the lack of the density of $C_0^\infty(\R^n)$,
In case $f\equiv f(x)$, Borchers and Miyakawa
\cite{Borchers Miyakawa AM1995}
refered to the existence of a strong solution of (N-S) with $u(t)\rightharpoonup a$ in weakly $*$ in $L^{n,\infty}(\Omega)$, as a solution of the
purterbed equations (P) below from the stationary solution $v$ associated with the force $f$:
\begin{equation*}\tag{P}
\frac{d}{dt}w -\Delta w +\PP[v\cdot\nabla w+w\cdot\nabla v]+\PP[w\cdot\nabla w]=0,
\qquad t>0,
\end{equation*}
which has apparently no forces. 
In \cite{Borchers Miyakawa AM1995},
they consider the stability in $L^{n,\infty}(\Omega)$
introducing the subspace $L_0^{n,\infty}(\Omega)$
of the completion 
of $C_0^\infty(\Omega)$ 
in $L^{n,\infty}(\Omega)$
where the Stokes semigroup is strongly continuous.
Recently, with the subspace $L_0^{n,\infty}(\Omega)$, Koba \cite{Koba JDE 2017} and 
Maremonti \cite{Maremonti RM2017} 
considered
 the existence of the strong solution
of (N-S) and (P), 
the stability and the uniqueness of mild solution of (N-S) without  \eqref{assump;unique}.

In case of non-trivial force $f=f(x,t)$, 
we need the essential treatment of the Duhamel terms
which comes from $f$.
Yamazaki \cite{Yamazaki MA2000}
consider the global existence and the stability of the weak mild solution of (N-S) in 
$L^{n,\infty}(\Omega)$ for small $a$ 
and $f=\nabla \cdot F$ with 
small $F(t)\in L^{\frac{n}{2},\infty}(\Omega)$. 
See also Definition \ref{def;weakmild} below.
On the other hand, our previous work \cite{Okabe Tsutsui periodic} construct a time periodic strong
solution in $BC\bigl(\R\,;\,
L^{n,\infty}(\R^n)\bigr)$ by
a different approach from \cite{Koba JDE 2017,Maremonti RM2017}, 
assuming a qualitative condition 
only on $f$ which satisfies
H\"{o}lder continuous on $\R$ 
with value in $L^{n,\infty}(\R^n)$ such as
\begin{equation*}\tag{A}
\lim_{\ep\searrow 0} \| e^{\ep\Delta}\PP f(t)-
\PP f(t)\|_{n,\infty}=0,
\qquad\text{for a.e. }t \in \R.
\end{equation*}

The aim of this paper is to
establish the global and the 
local well-posedness of the strong solvability
for the Cauchy problem of (N-S) with
non-trivial external forces.  
Firstly, we construct a global weak mild and mild 
solution
$u \in BC\bigl((0,\infty) \,;\,
L^{n,\infty}(\R^n)\bigr)$ of (N-S)
for small $a \in L^{n,\infty}(\R^n)$ 
and small $f \in BC\bigl([0,\infty)\,;\,
 L^{\frac{n}{3},\infty}(\R^n)\bigr)$, $n\geq 4$
and
$f \in BC\bigl([0,\infty)\,;\,L^{1}(\R^3)\bigr)$
which are scale invariant classes for initial data
and external forces, respectively.
We note that in this case the solution class
contains stationary solutions 
and time periodic solutions of (N-S).
Here, the key is the Meyer's estimate based on the 
$K$-method on $L^{n,\infty}(\R^n)$
which enables us to deal with  external forces
with the critical regularity. 
Then we observe this mild solution of (N-S)
becomes a strong solutions with the aid of (A).

Secondly, inspired by the condition (A) above,
we are successful to characterize the subspace
$X_\sigma^{n,\infty}$
in $L^{n,\infty}(\R^n)$ which is equivalent to the
condition (A).
Here we note that $X_{\sigma}^{n,\infty}$ 
is the maximal subspace 
where the Stokes semigroup 
is strongly continuous at $t=0$ and that
$X_{\sigma}^{n,\infty}$ is a strictly wider 
class than that in 
\cite{Koba JDE 2017,Maremonti RM2017}, 
see Remark \ref{rem;equivalent} below.
Lunardi \cite{Lunardi 1995} charachterized 
the subspace by the general theory 
of sectorial operators.
On the other hand, we give another proof 
for a different kind of operators.

Finally,
by the virtue of $X_{\sigma}^{n,\infty}$,
we establish the local well-posedness 
of the Cauchy problem of (N-S) 
in $X_{\sigma}^{n,\infty}$.
We construct a local  weak mild solution 
$u \in BC\bigl([0,T)\,;\,
X_{\sigma}^{n,\infty}\bigr)$
of (N-S) 
for every $a \in X_{\sigma}^{n,\infty}$
and $f \in BC\bigl([0,T)\,;\,
L^{\frac{n}{3},\infty}(\R^n)\bigr)$,
$n\geq 4$ and 
$f \in BC\bigl([0,T)\,;\,L^1(\R^3)\bigr)$
with less spatial singularity.
In this case, since $f$ has just critical regularity, there is a difficulty that
weak $L^n$-norm is only one which
is applicable to the iteration scheme.
Hence, as a different way from the usual 
Fujita-Kato (auxiliary norm) approach,
we introduce another iteration scheme
where $a\in X_{\sigma}^{n,\infty}$ is
much effective.
The existence of a local solution of (N-S)
yields the uniqueness of weak mild solution
in $BC\bigl([0,T)\,;\,L^{n,\infty}(\R^n)\bigr)$
as long as $a$ and $f$ have less singularity
within the scale critical spaces, respectively.

Moreover, in case  $u \in BC\bigl([0,T)\,;\,L^{n,\infty}(\R^n)\bigr)$, we see that
the initial data necessarily belong to 
$X_{\sigma}^{n,\infty}$.
Consequently, if $a$ has a bad singularity 
like $L^{n,\infty}(\R^n)\setminus X_{\sigma}^{n,\infty}$, then the weak mild solution of (N-S)
never happens to satisfy $\lim\limits_{t\to 0}u(t)=a$ in $L^{n,\infty}(\R^n)$ nor $u \in C\bigl((0,T)\,;\,L^r(\R^n)\bigr)$ with any $r>n$.

\medskip
This paper is organized as follow. In Section 2, 
we state our results. Section 3 is devoted to
the notations of the Lorenz spaces, to the critical estimate according to Meyer \cite{Meyer}
and to properties of an abstract evolution equations whose evolution semigroup is not strongly continuous at $t=0$, introduced by the previous work 
\cite{Okabe Tsutsui periodic}. 
In Section 4, we construct
 global weak mild and mild solutions of (N-S) by Fujita-Kato approach. In Section 5, we discuss the maximal subspace associated with the condition (A). 
In Section 6, we construct 
local weak mild and mild solutions of (N-S). 
In Section 7, we consider our global mild solution of (N-S) becomes strong solutions as a completion of the proof of Theorem \ref{thm;global}.
In Section 8, we consider local weak mild soltions
of (N-S) becomes strong solutions as a completion of the proof of Theorem \ref{thm;localX}
In Section 9 we consider the 
uniqueness criterion and give a proof of Theorem
\ref{thm;brezis}.%
\section{Results}
%
Before stating results, we introduce the following 
notations and some function spaces.
Let $C_{0,\sigma}^\infty(\mathbb{R}^n)$ denotes 
the set of all $C^\infty$-solenoidal vectors $\phi$ with compact support
in $\mathbb{R}^n$, i.e., $\Div \phi=0$ in $\mathbb{R}^n$.
$L^r_{\sigma}(\mathbb{R}^n)$ is the closure of 
$C^\infty_{0,\sigma}(\mathbb{R}^n)$ with respect to the $L^r$-norm 
$\|\cdot \|_r$, $1< r < \infty$.
$(\cdot,\cdot)$ is the duality pairing between $L^r(\mathbb{R}^n)$
and $L^{r^\prime}(\mathbb{R}^n)$, where $1/r+1/r^\prime=1$, 
$1\leq r < \infty$.
$L^r(\mathbb{R}^n)$ and $W^{m,r}(\mathbb{R}^n)$ denote 
the usual (vector-valued) $L^r$-Lebesgue space and $L^r$-Sobolev space over 
$\mathbb{R}^n$, respectively. 
Moreover, $\mathcal{S}(\R^n)$ denotes the set of all of the Schwartz functions. $\mathcal{S}^\prime(\R^n)$ denotes the set of all tempered distributions.
When $X$ is a Banach space, $\| \cdot \|_{X}$ denotes the norm on $X$.
Moreover, $C(I;X)$, $BC(I;X)$ and $L^r(I;X)$ denote 
the $X$-valued continuous and bounded continuous functions over the interval $I\subset \R$, 
and $X$-valued $L^r$ functions, respectively.

Moreover, for $1 < p < \infty$ and $1\leq q \leq \infty$ let 
$L^{p,q}(\R^n)$ be the space of all locally integrable functions with (quasi) norm $\|f\|_{p,q} < \infty$, where
\begin{equation*}
\|f\|_{p,q}= \begin{cases}\displaystyle
\left(\int_0^\infty 
\bigl(\lambda \,|\{x\in \R^n \,;\, |f(x)|>\lambda \}|^{\frac{1}{p}}
\bigr)^{q}\frac{d\lambda}{\lambda}\right)^{\frac{1}{q}},
& 1\leq q <\infty,\\
\displaystyle
\sup_{\lambda>0} \lambda \,|\{x\in \R^n\,;\, |f(x)|>\lambda\}|^{\frac{1}{p}},
& q=\infty,
\end{cases}
\end{equation*}
where $|E|$ denotes the Lebesgue measure of $E \subset \R^n$.
For the case $q=\infty$, 
$L^{p,\infty}(\R^n)$ is a Banach space with the following norm:
with any $1 \le r < p$
\begin{equation*}
\| f \|_{L^{p,\infty}}=\sup_{0<|E|<\infty} |E|^{-\frac{1}{r}+\frac{1}{p}}
\Bigl(\int_E |f(x)|^r \,dx\Bigr)^{\frac{1}{r}}.
\end{equation*}
Here, we note that $\|\cdot\|_{L^{n,\infty}}$ 
is equivalent to $\|\cdot\|_{n,\infty}$. 
Since $\PP$ is a bounded operator on $L^{p,\infty}(\R^n)$ for $1<p<\infty$, 
we introduce
the set of solenoidal vectors in $L^{p,\infty}(\R^n)$ as 
$L^{p,\infty}_{\sigma}(\R^n)=\PP L^{n,\infty}(\R^n)$.

\medskip

\begin{definition}[Weak mild solution]
\label{def;weakmild}
Let $a \in L_\sigma^{n,\infty}(\R^n)$ 
and $f\in BC\bigl([0,\infty)\,;\,
L^{p,\infty}(\R^n)\bigr)$
for some $n/3\leq p \leq n$.
We call a function
$u \in BC\bigl((0,\infty)\,;\,
L_{\sigma}^{n,\infty}(\R^n)\bigr)$
weak (generalized) mild solution of (N-S),
if
\begin{equation*}\tag{IE$^{*}$}
u(t)=e^{t\Delta }a
+ \int_0^t e^{(t-s)\Delta}\PP f(s)\,ds
-\int_0^t 
\nabla \cdot e^{(t-s)\Delta}
\PP (u\otimes u)(s)\,ds, 
\qquad 0<t<T.
\end{equation*}
\end{definition}
\begin{remark}\label{rem;weakmildsol}\rm
In case of $n=3$, we modify the condition as 
$f \in BC([0,\infty)\,;\,L^{1}(\R^3))$ and
\begin{equation*}\tag{IE$^{**}$}
u(t)=e^{t\Delta }a
+ \int_0^t \PP e^{(t-s)\Delta} f(s)\,ds
-\int_0^t 
\nabla \cdot e^{(t-s)\Delta}
\PP (u\otimes u)(s)\,ds, 
\qquad 0<t<T.
\end{equation*}
Moreover, a weak mild solution $u$ satisfies
\begin{equation*}
\bigl(u(t),\varphi\bigr)=
(e^{t\Delta}a,\varphi )
+
\int_0^t \bigl(e^{(t-s)\Delta} f(s),
\varphi\bigr)\,ds
+
\int_0^t\bigl(u(s)\cdot \nabla e^{(t-s)\Delta}\varphi,u(s)\bigr)\,ds, \quad 0<t<T , \quad
\varphi \in C_{0,\sigma}^{\infty}(\R^n).
\end{equation*}
See also, Kozono and Yamazaki \cite{Kozono Yamazaki HJM1995}, Yamazaki \cite{Yamazaki MA2000}.
\end{remark}
\begin{definition}[Mild solution]
Let $a \in L^{n,\infty}_{\sigma}(\R^n)$ 
and $f\in BC\bigl([0,T)\,;\,L^{p,\infty}(\R^n)\bigr)$ 
for some $n/3\leq p\leq n$.
Then a function 
$u \in BC\bigl((0,T)\,
;\,L_{\sigma}^{n,\infty}(\R^n)\bigr)$ 
which satisfies
$\nabla u \in C\bigl((0,T)\,;\,L^{q,\infty}(\R^n)\bigr)$
with $\limsup\limits_{t\to 0}t^{1-\frac{n}{2q}}\|\nabla u(t)\|_q<\infty$
for some $q\geq n/2$ is called a mild solution of (N-S),
if
\begin{equation*}\tag{IE}
u(t)=e^{t\Delta}a
+
\int_0^t e^{(t-s)\Delta}\PP f(s)\,ds
-
\int_0^t e^{(t-s)\Delta}\PP [u\cdot\nabla u](s)\,ds,
\qquad 0<t<T.
\end{equation*}
\end{definition}
\begin{remark}\rm
In case of $n=3$, we introduce a similar modification for $f$ as in Remark \ref{rem;weakmildsol}. 
We note that $u(t)$ tends to $a$ as $t\searrow0$
in the sense of distributions, i.e.,
$\bigl(u(t),\varphi\bigr)\to(a,\varphi)$ as
$t\searrow 0$ for all $\varphi \in C_0^{\infty}(\R^n)$.
Moreover, additionally if
$\int_0^t e^{(t-s)\Delta}\PP f(s)\,ds
\rightharpoonup 0$ weakly $*$ in $L^{n,\infty}(\R^n)$ as $t\to0$ and if  
$u \in BC\bigl((0,T)\,;\,L_{\sigma}^r(\R^n)\bigr)$ with some $r>n$ or 
$\nabla u \in C\bigl((0,T)\,;\,L^q(\R^n)\bigr)$ with 
$\limsup\limits_{t\to 0}t^{1-\frac{n}{2q}}\|\nabla u(t)\|_q =0$
for some $q >n/2$, then it holds that $u(t) \rightharpoonup a$ weakly $*$ in $L^{n,\infty}(\R^n)$ as $t \searrow 0$.
However, we are unable to obtain $u(t) \to a$ in $L^{n,\infty}$ as $t\searrow 0$ in general, since $\{e^{t\Delta}\}$ is not strongly 
continuous at $t=0$ in $L^{n,\infty}(\R^n)$.
\end{remark}
\begin{definition}[Strong solution]\label{defi:strong}
Let $a \in L_{\sigma}^{n,\infty}(\R^n)$ and 
$f\in BC\bigl([0,T)\,;\,L^{n,\infty}(\R^n)\bigr)$
Then a function $u$ is called
a strong solution of (N-S), if
\begin{itemize}
\item[(i)] 
$u \in BC\bigl((0,T)\,;\,L_{\sigma}^{n,\infty}(\R^n)\bigr)\cap C^1\bigl((0,T)\,;\,L_{\sigma}^{n,\infty}(\R^n)\bigr)$,
\item[(ii)] 
$u(t)\in \{u\in L_{\sigma}^{n,\infty}\,;\,\Delta u \in L^{n,\infty}(\R^n)\}$ for $0<t<T$ and 
$\Delta u \in C\bigl((0,T)\,;\, L^{n,\infty}(\R^n)\bigr)$,
\item[(iii)] $u$ satisfies (N-S) in the following sense.
\begin{equation*}\tag{DE}
\left\{\begin{split}
&\frac{du}{dt}-\Delta u + \PP [u\cdot\nabla u]=\PP f
\quad\text{in } L_{\sigma}^{n,\infty}(\R^n),\quad 0<t<T,
\\
&\lim_{t\to\infty}\bigl(u(t),\varphi\bigr)=(a,\varphi)
\quad \text{for all } \varphi\in C_{0}^{\infty}(\R^n).
\end{split}\right.
\end{equation*}
\end{itemize}
\end{definition}
\begin{remark}\rm
The strong solution $u$ in the class (i) and (ii) in Definition \ref{defi:strong} necessarily satisfies
the initial condition in the sense of distributions.
Indeed, the strong solution $u$ satisfies (IE).
Then, noting that 
for each $t$, $u(t) \in L^r(\R^n)$ for some $r>n$ 
by the Sobolev embedding and by the real interpolation, we see that for 
$\phi\in C_{0,\sigma}^{\infty}(\R^n)$ and for $p>n/(n-1)$
\begin{equation*}
\begin{split}
\left|\left(\int_0^t 
e^{(t-s)\Delta}\PP[u\cdot\nabla u](s)\,ds,
\phi\right)\right|
&=\left|-\int_0^t
\bigl(u(s)\cdot\nabla e^{(t-s)\Delta}\phi,
u(s)\bigr)ds\right|
\\
&\leq t^{-\frac{n}{2p}+\frac{n}{2}-\frac{1}{2}}
 \sup_{0<s<T}\|u(s)\|_{n,\infty}^2
\|\phi\|_p \\
& \to 0 \quad\text{as } t \searrow 0,
\end{split}
\end{equation*}
since $-\frac{n}{2p}+\frac{n}{2}-\frac{1}{2}>0$.
\end{remark}
Since $\{e^{t\Delta}\}$ is not strongly continuous, 
it is open problem that a mild solution of (N-S) 
satisfies (DE) in the topology of $L^{n,\infty}(\R^n)$. 
In our previous work \cite{Okabe Tsutsui periodic}, 
we introduce the condition (A) for the external forces 
which enables us to verify 
that a mild solution of (N-S) becomes a strong solution. 
With this condition, 
we investigate the global in time 
and also local in time strong solvability of (N-S).

\begin{theorem}\label{thm;global}
(i) Let $n\geq 4$. There exists $\ep_n>0$ such that
for $a \in L_{\sigma}^{n,\infty}(\R^n)$ and $f \in BC\bigl([0,\infty)\,;\,
L^{\frac{n}{3},\infty}(\R^n)\bigr)$.
 if 
\begin{equation*}
\| a \|_{n,\infty} +
\sup_{0<t<\infty}\|\PP f(t)\|_{\frac{n}{3},\infty} < \ep_n,
\end{equation*}
then there exists a global in time weak mild solution 
$u\in BC\bigl((0,\infty)\,;\, L^{n,\infty}(\R^n)\bigr)$ of (N-S). 

(ii) For $n\geq 4$ there exists $\ep_n^\prime>0$
with the following property. 
If $a \in L_{\sigma}^{n,\infty}(\R^n)$ also satisfies $\nabla a \in L^{\frac{n}{2},\infty}(\R^n)$ and
\begin{equation*}
\|\nabla a\|_{\frac{n}{2},\infty} 
+ 
\sup_{0<t<\infty}\|\PP f(t)\|_{\frac{n}{3},\infty}
<\ep_n^\prime
\end{equation*}
then there exists a mild solution $u$ of (N-S) in the class 
\begin{equation*}
u \in BC\bigl((0,\infty)\,;\,
L_{\sigma}^{n,\infty}(\R^n)\bigr)
\quad\text{and}\quad
\nabla u \in BC\bigl( (0,\infty)\,;\,L^{\frac{n}{2},\infty}(\R^n)\bigr).
\end{equation*}

(iii) For $n\geq 4$ and $\frac{n}{3}<p<\frac{n}{2}$,
there exists $0<\ep_{n,p}\leq\ep_{n}^\prime$
with the the following property.
Let $a \in L^{n,\infty}_{\sigma}(\R^n)$ satisfy
$\nabla a \in L^{\frac{n}{2},\infty}(\R^n)
\cap L^{q,\infty}(\R^n)$ with $\frac{1}{q}=\frac{1}{p}-\frac{1}{n}$
and let $f \in BC\bigl([0,\infty)\,;\,
L^{\frac{n}{3},\infty}(\R^n)\cap L^{p,\infty}(\R^n)\bigr)$. If 
\begin{equation*}
\| \nabla a \|_{\frac{n}{2},\infty}+
\sup_{0<t<\infty}\|\PP f(t)\|_{\frac{n}{3},\infty}
< 
\ep_{n,p},
\end{equation*}
then the mild solution $u$ obtained by (ii) above satisfies
\begin{equation*}
u \in BC\bigl((0,\infty)\,;\,
L_{\sigma}^{r,\infty}(\R^n)\bigr)
\quad\text{and}\quad
\nabla u \in BC\bigl( (0,\infty)\,;\,L^{q,\infty}(\R^n)\bigr),
\end{equation*} 
where $\frac{1}{r}=\frac{1}{q}-\frac{1}{n}=\frac{1}{p}-\frac{2}{n}$, 
and consequently satisfies $u(t)\rightharpoonup a$ weakly $*$ in $L_{\sigma}^{n,\infty}(\R^n)$
as $t\searrow 0$.

(iv) We additionally assume 
that $\PP f$ is H\"{o}lder continuous on $[0,\infty)$ 
in $L_{\sigma}^{n,\infty}(\R^n)$ and that
\begin{equation*}\tag{A}
\lim_{\ep\to0} \|e^{\ep\Delta}\PP f(t)-\PP f(t)\|_{n,\infty}=0, 
\end{equation*}
for almost every $t\in [0,\infty)$.
Then we have that 
a mild solution $u$ obtained by (iii) above is a strong solution of (N-S)
with $u(t)\rightharpoonup a$ 
weakly $*$ in $L_{\sigma}^{n,\infty}(\R^n)$ 
as $t\searrow 0$.

(v) Let $n=3$.
There exists $\ep_3>0$ such that for $a \in L_{\sigma}^{3,\infty}(\R^3)$ and $f\in BC\bigl([0,\infty)\,;\,L^1(\R^3)\bigr)$ if 
\begin{equation*}
\|a \|_{3,\infty}
+
\sup_{0<t<\infty} \|f(t)\|_1 < \ep_3,
\end{equation*}
then there exists a weak mild solution $u \in BC\bigl((0,\infty)\,;\, L_{\sigma}^{3,\infty}(\R^3)\bigr)$ with (IE$^{**}$).

(vi) For $n=3$ and $1<p<\frac{3}{2}$, there exists 
$0<\ep_{3,p}\leq \ep_3$ with the following property. 
Let $a \in L_\sigma^{3,\infty}(\R^3)$ satisfy
$\nabla a \in L^{\frac{3}{2},\infty}(\R^3)
\cap L^{q,\infty}(\R^3)$ with $\frac{1}{q}=\frac{1}{p}-\frac{1}{3}$
and let $f\in BC\bigl([0,\infty)\,;\,
L^{1}(\R^3)\cap L^{p,\infty}(\R^3)\bigr)$.
If 
\begin{equation*}
\| \nabla a \|_{\frac{3}{2},\infty}
+
\sup_{0<t<\infty}\|f(t)\|_{1}
<\ep_{3,p},
\end{equation*}
then the weak mild solution $u$ obtained by (v) is
the mild solution of (N-S) in the class
\begin{equation*}
u \in BC\bigl((0,\infty)\,;\,
L_{\sigma}^{r,\infty}(\R^n)\bigr)
\quad\text{and}\quad
\nabla u \in BC\bigl( (0,\infty)\,;\,L^{q,\infty}(\R^n)\bigr),
\end{equation*} 
where $\frac{1}{r}=\frac{1}{q}-\frac{1}{n}=\frac{1}{p}-\frac{2}{n}$.
Further, $u(t)\rightharpoonup a$ weakly $*$
in $L^{3,\infty}_{\sigma}(\R^3)$ as $t\searrow 0$.

(vii) Similarly, we additionally assume 
that $\PP f$ is H\"{o}lder continuous on $[0,\infty)$
in $L_{\sigma}^{3,\infty}(\R^3)$ and that
$\PP f$ satisfies (A) for almost every $t\geq 0$. 
Then 
the mild solution $u$ obtained by (vi) above becomes the strong solution of (N-S)
with $u(t)\rightharpoonup a$ 
weakly $*$ in $L_{\sigma}^{3,\infty}(\R^3)$ 
as $t\searrow 0$.
\end{theorem}
\begin{remark}\rm
(i) We note that the condition (A) is required only for the force $\PP f$. Since we need no restriction on the initial data in $L_{\sigma}^{n,\infty}(\R^n)$, 
we are able to handle the function like $1/|x|$ as an initial data.

(ii) In the assumption in Theorem \ref{thm;global},
time decay of the forces is not required. 
Hence, our class of solutions contains 
the stationary solutions and the time periodic 
solutions of (N-S) 
if the forces are independent of time and are periodic in time, respectively.
So we may not expect the time decay property for the solution obtained by Theorem \ref{thm;global} in general. 
Of course, if we replace the smallness condition by
\begin{equation*}
\| a \|_{n,\infty} 
+
\|\PP f\|_{L^1(0,\infty;L^{n,\infty})}
+
\sup_{0<t<\infty}t \|\PP f(t)\|_{n,\infty}
<\ep_n,
\end{equation*}
then we obtain a global strong solution $u$ with $\|u(t)\|_r\to 0$ as $t\to \infty$, for $r>n$.

(iii) Theorem \ref{thm;global} does not guarantee
that \textit{every} mild solution of (N-S) becomes
a strong solution of (N-S).
Indeed, the key of the proof based on the existence and the uniqueness of local solution of (N-S). 
So additional regularity is needed for the initial data and the mild solution of (N-S).
See, Theorem \ref{thm;localX} and Theorem \ref{thm;local weak uniqueness} below.  
\end{remark}
In Theorem \ref{thm;global}, the condition (A) plays an
important role for the strong solvability of (N-S)
in $L_{\sigma}^{n,\infty}(\R^n)$. 
The following theorem characterizes the functions
which satisfies the condition (A). 
For this purpose, we introduce the domain of the Stokes operator $-\Delta$ in $L_{\sigma}^{n,\infty}(\R^n)$ as
$D(-\Delta)=\{u\in L_{\sigma}^{n,\infty}(\R^n)\,;\,
\Delta u \in L^{n,\infty}(\R^n)\}$.
We find the following theorem in Lunardi
\cite{Lunardi 1995}, but we give another proof.
\begin{theorem}[{Lunardi \cite{Lunardi 1995}}]\label{thm;maximal}
Let $f\in L_{\sigma}^{n,\infty}(\R^n)$. Then it holds that
\begin{equation*}
\lim_{\ep\searrow 0}\| e^{\ep\Delta}f-f\|_{n,\infty}=0
\quad\text{if and only if}\quad
f \in \overline{D(-\Delta)}^{\|\cdot\|_{n,\infty}}.
\end{equation*}
Consequently, $\{e^{t\Delta}\}_{t\geq 0}$ is a bounded $C_0$-analytic semigroup on 
$\overline{D(-\Delta)}^{\|\cdot\|_{n,\infty}}$.
In other words,
$\overline{D(-\Delta)}^{\|\cdot\|_{n,\infty}}$
is the maximal subspace in $L_{\sigma}^{n,\infty}(\R^n)$
where the Stokes semigroup is $C_0$-semigroup. 
\end{theorem}
\begin{remark}\label{rem;equivalent}\rm
(i) We note that $\overline{C_{0,\sigma}^\infty(\R^n)}^{\|\cdot\|_{n,\infty}} \subsetneqq \overline{D(-\Delta)}^{\|\cdot\|_{n,\infty}}$. 
Indeed, take $f(x)\sim 1/|x|$ for $|x|\gg1$.
Then we see that $f \in D(-\Delta)$, 
but $f \notin \overline{C_{0,\sigma}^\infty(\R^n)}^{\|\cdot\|_{n,\infty}}$. 

(ii) The condition (A) is not a necessary 
condition for external forces 
for the strong solvability of the Stokes equations
and the Naiver-Stokes equations 
in $L_{\sigma}^{n,\infty}(\R^n)$. 
Indeed, take 
$f\in 
\bigl(L^{n,\infty}(\R^n)\setminus \overline{D(-\Delta)}^{\|\cdot\|_{n,\infty}}\bigr)\cap L^{\frac{n}{3},\infty}(\R^n)$ for $n\geq 4$ 
and consider $u=(-\Delta)^{-1}\PP f$ for the Stokes equations and
 $u=(-\Delta)^{-1}\PP f 
 - (-\Delta)^{-1}\PP [u\cdot \nabla u]$ 
 for the Navier-Stokes equations. 
 Then we see that $u \in D(-\Delta)$ and satisfies
 the equations for the strong sense.
 
 (iii) With our method to prove the Theorem \ref{thm;maximal}, 
in a general Banach space $X$,
for every bounded analytic semigroup
 $\{e^{tL}\}$ on $X$ with the property that
 $e^{tL}a$ is weakly or weakly$*$
 continuous at $t=0$ for all $a\in X$,
we also characterize the
 maximal subspace as $\overline{D(L)}^{X}$ where $\{e^{tL}\}$ is strongly continuous.
Here we never use the theory of sectorial operators as a different approach from Lunardi \cite{Lunardi 1995}.
\end{remark}
%
Next, we consider the existence of a local in time solution of (N-S) by the virtue of 
the subspace in Theorem \ref{thm;maximal}.
So we shall introduce a notation such as $X_{\sigma}^{n,\infty}:=
\overline{D(-\Delta)}^{\|\cdot\|_{n,\infty}}$.
On the other hand, for local existence of a weak mild or a mild solution, 
$f\in BC\bigl([0,T)\,;\,
L^{\frac{n}{3},\infty}(\R^n)\bigr)$ is not enough.
We restrict  $f(t)\in 
\widetilde{L}^{\frac{n}{3},\infty}(\R^n)
=\overline{L^{\frac{n}{3},\infty}(\R^n)\cap L^{\infty}(\R^n)}^{\|\cdot\|_{\frac{n}{3},\infty}}$ for $t\geq 0$
as a treatment of a spatial singularity of the force. For this space see, for instance,
Farwig and Nakatsuka and Taniuchi 
\cite{Farwig Nakatsuka Taniuchi 1,
Farwig Nakatsuka Taniuchi 2}.
\begin{theorem}\label{thm;localX}
Let $n\geq 3$ and $a \in X_{\sigma}^{n,\infty}$.

(i) Suppose
$f\in BC\bigl([0,\infty)\,;\, 
\widetilde{L}^{\frac{n}{3},\infty}(\R^n)\bigr)$ for $n\geq 4$ and 
$ f \in BC\bigl([0,\infty);{L}^{1}(\R^3)\bigr)$.
Then there exist $T>0$ and a weak mild solution 
$u \in BC\bigl([0,T)\,;\,X_{\sigma}^{n,\infty}\bigr)$
of (N-S) with
\begin{equation*}
u(t)\to a \quad\text{in }L_{\sigma}^{n,\infty} \quad
\text{as }t\searrow 0.
\end{equation*}

(ii) Suppose $f\in BC\bigl([0,\infty)\,;\,L^{p,\infty}(\R^n)\bigr)$ with some 
$\frac{n}{3}<p\leq n$. 
Then there exist $T>0$ and a weak mild solution 
$u\in BC\bigl([0,T)\,;\,
X_{\sigma}^{n,\infty}\bigr)$ 
of (N-S) with $u(t)\to a$
in $L^{n,\infty}(\R^n)$ as $t\searrow 0$.

(iii) Furthermore, if additionally $\PP f$ is H\"{o}lder continuous
on $[0,T)$ in $L^{n,\infty}_{\sigma}(\R^n)$ and satisfies
(A), i.e., $\PP f(t)\in X_{\sigma}^{n,\infty}$ 
for almost every $0<t<T$, 
then the weak mild solution $u$ 
obtained by (i) or (ii) above becomes the strong solution of (N-S) with
$u(t)\to a$ in $L_{\sigma}^{n,\infty}(\R^n)$ as $t\searrow 0$.
\end{theorem}
\begin{remark}\rm
(i) If $n\geq 4$ and 
$\nabla a \in L^{\frac{n}{2},\infty}(\R^n)$, 
the weak mild solution $u$
obtained by (i) of Theorem \ref{thm;localX} is
actually a mild solution with 
$\nabla u \in BC\bigl([0,T)\,;\,L^{\frac{n}{2},\infty}(\R^n)\bigr)$.
Similarly, if $p=n$ or $\nabla a \in L^{q,\infty}(\R^n)$ with 
$\frac{1}{q}=\frac{1}{p}-\frac{1}{n}$,
$\frac{n}{3}<p<n$, then the weak mild solution
$u$ obtained by (ii) of Theorem \ref{thm;localX}
is actually a mild solution of (N-S).

(ii) The solution class $BC\bigl([0,T)\,;\, 
X_{\sigma}^{n,\infty}\bigr)$ is well known 
for the uniqueness of weak mild or mild solutions
of (N-S) since the Stokes semigroup is strongly continuous  for $t\geq 0$.

(iii) For the local in time solvability, 
$a \in \widetilde{L}_{\sigma}^{n,\infty}(\R^n)=\overline{L_{\sigma}^{n,\infty}(\R^n)\cap L^{\infty}(\R^n)}^{\|\cdot\|_{n,\infty}}$ 
and
$f \in BC\bigl([0,T)\,;\, L^{p,\infty}(\R^n)\bigr)$ for $\frac{n}{3}<p\leq n$
are also valid.
Since $X_{\sigma}^{n,\infty} \subset \widetilde{L}_{\sigma}^{n,\infty}(\R^n)$, 
$\widetilde{L}_{\sigma}^{n,\infty}(\R^n)$
is a wider class of initial data $a$ for  local weak mild or mild solutions 
$u\in BC\bigl((0,T)\,;\, 
L_{\sigma}^{n,\infty}(\R^n)\bigr)$ 
of (N-S) with $u(t)\rightharpoonup a$ weakly $*$ in $L^{n,\infty}(\R^n)$.

(iv) Borchers and Miyakawa \cite{Borchers Miyakawa AM1995}, and Koba \cite{Koba JDE 2017} consider 
the stability of the stationary solution of (N-S) in
$L_{0, \sigma}^{n,\infty}(\Omega)=
\overline{C_{0,\sigma}^{\infty}(\Omega)}^{\|\cdot\|_{n,\infty}}$. In the framework of $L_{0,\sigma}^{n,\infty}(\Omega)$, we expect that the asymptotic stability of 
the solution $u$ in the critical norm, i.e.,
$\lim\limits_{t\to\infty}\|u(t)\|_{n,\infty}=0$.
However, $X_{\sigma}^{n,\infty}$ seems not to allow the
time decay of the solution with the $L^{n,\infty}$-norm.
\end{remark}

The uniqueness theorem is expected within the 
solution class such as $BC\bigl([0,T)\,;\,
X_{\sigma}^{n,\infty}\bigr)$, 
see, for instance, \cite{Tsutsui AJDE 2011}.
Therefore,
the natural question is that 
when  $u\in BC\bigl([0,T)\,;\,
X_{\sigma}^{n,\infty}\bigr)$.
The following theorem implies that
if the singularity of data are well-controled, then the orbit of the solution is unique and stays in $X_{\sigma}^{n,\infty}$. 
\begin{theorem}\label{thm;local weak uniqueness}
Let $n\geq 3$ and let $a\in X_{\sigma}^{n,\infty}$
 and $f \in BC\bigl([0,T)\,;\, 
 \widetilde{L}^{\frac{n}{3},\infty}(\R^n)\bigr)$
 for $n\geq 4$ and
 $f\in BC\bigl([0,T)\,;\,L^{1}(\R^3)\bigr)$.
 Suppose $u$, $v$ are two 
 weak mild solutions of (N-S)
 with $u|_{t=0}=v|_{t=0}=a$. If
 \begin{equation*}
 u,v \in BC\bigl([0,T)\,;\,
 L_{\sigma}^{n,\infty}(\R^n)\bigr),
 \end{equation*}
 then it holds 
 \begin{equation*}
 u,v \in BC\bigl([0,T)\,;\, 
 X_{\sigma}^{n,\infty}\bigr)
 \quad\text{and}\quad
 u\equiv v.
 \end{equation*}
\end{theorem}
\begin{remark}\rm
(i) Suppose $a \in X_{\sigma}^{n,\infty}$,
$f \in BC\bigl([0,\infty)\,;\,
\widetilde{L}^{\frac{n}{3},\infty}(\R^n)\bigr)$
if $n\geq 4$ and $f \in BC\bigl([0,\infty)
\,;\,L^1(\R^3)\bigr)$.
Then the weak mild solution $u$ of (N-S) 
obtained by
Theorem \ref{thm;global} satisfies
$u \in BC\bigl([0,\infty)\,;\, X_{\sigma}^{n,\infty}\bigr)$ and is unique without any smallness.

(ii) Similarly, for $a\in X_{\sigma}^{n,\infty}$,
the correspond mild solution $u$ of (N-S)
obtained Theorem \ref{thm;global} satisfies
$u\in BC\bigl([0,\infty)\,;\,X_{\sigma}^{n,\infty}\bigr)$ and is unique.
\end{remark}
In the above theorem, we restrict the singularity
of data. 
Next, we shall give the uniqueness criterion 
for wider class of data, especially, for general initial data.
For this purpose, we focus on the continuity
$\lim\limits_{t\to0}u(t)=a$ in 
$L^{n,\infty}(\R^n)$.
Furthermore, in such a case we find a
necessary condition for the initial data.

\begin{theorem}\label{thm;brezis}
(i) Let $a \in L_{\sigma}^{n,\infty}(\R^n)$ and 
$f\in BC\bigl([0,T)\,;\,\widetilde{L}^{\frac{n}{3},\infty}(\R^n)\bigr)$, $n\geq 4$ and
$f\in BC\bigl([0,T)\,;\,L^1(\R^3)\bigr)$.
Further let $u,v$ be two weak mild solutions of (N-S) with $u|_{t=0}=v|_{t=0}=a$ in the class
\begin{equation*}
u\in BC\bigl([0,T)\,;\, 
L_{\sigma}^{n,\infty}(\R^n)\bigr)
\quad\text{and}\quad
v\in BC\bigl([0,T)\,;\, 
\widetilde{L}_{\sigma}^{n,\infty}(\R^n)\bigr).
\end{equation*}
Then $a\in X_{\sigma}^{n,\infty}$
and $u\equiv v$.

(ii) Let 
$a\in \widetilde{L}_{\sigma}^{n,\infty}(\R^n)$
and $f\in BC\bigl([0,T)\,;\,
\widetilde{L}^{\frac{n}{3},\infty}(\R^n)\bigr)$, $n\geq 4$ and 
$f\in BC\bigl([0,T)\,;\, L^1(\R^3)\bigr)$.
Suppose $u,v$ be two weak mild solutions of (N-S) with $u|_{t=0}=v|_{t=0}=a$ in the class
\begin{equation*}
u,v\in BC\bigl([0,T)\,;\, 
L_{\sigma}^{n,\infty}(\R^n)\bigr).
\end{equation*}
Then $a\in X_{\sigma}^{n,\infty}$
and $u\equiv v$.
\end{theorem}
\begin{remark}\rm
(i) The uniqueness criterion within the class
$BC\bigl([0,T)\,;\, \widetilde{L}_{\sigma}^{n,\infty}(\R^n)\bigr)$ is proved by Farwig and Taniuchi \cite{Farwig Taniuchi 2013} and 
Farwig and Nakatsuka and Taniuchi 
\cite{Farwig Nakatsuka Taniuchi 1,Farwig Nakatsuka Taniuchi 2}, see also Lions and Masmoudi 
\cite{Lions Masmoudi}.

(ii) Since $BC\bigl([0,T)\,;\,
L_{\sigma}^{n,\infty}(\R^n)\bigr)
\cap BC\bigl( [0,T)\,;\,
L^r(\R^n)\bigr)\subset
BC\bigl([0,T)\,;\,
\widetilde{L}_{\sigma}^{n,\infty}(\R^n)\bigr)$,
$a \notin X_{\sigma}^{n,\infty}$ implies that
$\lim\limits_{t\to0}u(t)\neq a$ in 
$L^{n,\infty}(\R^n)$ nor $u(t) \notin L^r(\R^n)$ for any $r>n$ if we have 
$u \in BC\bigl((0,T)\,;\,
\widetilde{L}_{\sigma}^{n,\infty}(\R^n)\bigr)$.
Moreover, Theorem \ref{thm;brezis} implies
that there is no weak mild solution
in $BC\bigl([0,T)\,;\,
L_\sigma^{n,\infty}(\R^n)\bigr)$
for $a\in \widetilde{L}_{\sigma}^{n,\infty}(\R^n)
\setminus X_{\sigma}^{n,\infty}$
if $X_{\sigma}^{n,\infty}\neq\widetilde{L}_{\sigma}^{n,\infty}(\R^n)$.  

(iii) For the case $f\equiv 0$, the uniqueness criterion in 
$BC\bigl([0,T)\,;\,
{L}_{\sigma}^{n,\infty}(\R^n)\bigr)$ for
$a$ which satisfies $e^{t\Delta}a \to a$ in $L^{n,\infty}(\R^n)$ as $t\to 0$ is well-known, see, for instance, the second author \cite{Tsutsui AJDE 2011}. However, the property
$\lim\limits_{t\to 0}e^{t\Delta}a=a$ is
actually derived through the behavior of the nonlinear term.
\end{remark}
%
%
\section{Preliminaries}
%
We firstly introduce the fundamental facts of the Lorentz spaces in terms of the real interpolation and
provide the theory of the Stokes semigroup on Lorentz spaces.
Next, we prepare the estimate of the Duhamel term of (N-S) with $L^{n,\infty}$-norm.
Finally, we refer to the strong solvability of 
abstract evolution equations on the Banach space,
where the associated semigroup 
is not strong continuous at $t=0$, 
introduced by our previous work 
\cite{Okabe Tsutsui periodic}.
%
\subsection{Lorentz spaces}
It is well known that the Lorentz space $L^{p,q}(\R^n)$ can be 
realized as a real interpolation space
between $L^{p_0}(\R^n)$ and $L^{p_1}(\R^n)$
with $0<p_0<p<p_1$, i.e., 
\begin{equation*}
L^{p,q}(\R^n)=\bigl( L^{p_0}(\R^n),L^{p_1}(\R^n)\bigr)_{\theta,q}, 
\end{equation*}
where $0<\theta<1$ satisfies 
$1/p ={(1-\theta)}/{p_0}+{\theta}/{p_1}$ and $1\leq q \leq \infty$.
We remark that if $1\leq q <\infty$ then $C_{0}^\infty(\R^n)$ is dense 
in $L^{p,q}(\R^n)$, $1\leq p < \infty$. 
Moreover, by the real interpolation, if $1\leq p <\infty$ and 
$1\leq q <\infty$, the dual space of $L^{p,q}(\R^n)$ is 
$L^{p^\prime,q^\prime}(\R^n)$ where 
${1}/{p}+{1}/{p^\prime}=1$ and 
${1}/{q}+{1}/{q^\prime}=1$.
See, Bergh and L\"{o}fstr\"{o}m \cite{Bergh Lofstrom},
Grafakos \cite{Grafakos}.
Furthermore we remark that if $1<p<\infty$ and $1\leq q \leq \infty$,
the Lorentz space $L^{p,q}(\R^n)$ may be regarded the Banach space, 
equipped with a norm which is equivalent to the quasi norm $\|\cdot\|_{p,q}$.
See, Grafakos \cite[Exercise 1.4.3]{Grafakos}.

In the whole space case, $\PP =( \delta_{j,k} I+ R_j R_k )_{1 \le j,k \le n}$, 
where $R_j := \partial_j(-\Delta)^{-1/2}$ is the Riesz transform.
Since the Riesz transforms or $\PP$ in $L^{p_0}(\R^n)$ coincide with 
those on $L^{p_1}(\R^n)$ over $L^{p_0}(\R^n)\cap L^{p_1}(\R^n)$ 
for $1<p_0<p_1<\infty$, $\PP$ is extended on $L^{p_0}(\R^n)+L^{p_1}(\R^n)$
and is bounded on $L^{p,q}(\R^n)$ for $p_0<p<p_1$ and $1\leq q \leq \infty$.
Furthermore, since $\PP$ is bounded and surjective 
form $L^p(\R^n)$ onto $L_{\sigma}^p(\R^n)$,
Miyakawa and Yamada \cite[Lemma 1.1]{Miyakawa Yamada HMJ1992} introduced
for $1<p<\infty$,
\begin{equation*}
\PP L^{p,\infty}(\R^n)=L^{p,\infty}_{\sigma}(\R^n)=
\bigl(L_{\sigma}^{p_0}(\R^n), L_{\sigma}^{p_1}(\R^n) \bigr)_{\theta, \infty},
\end{equation*}
where $1<p_0<p<p_1<\infty$ and $0<\theta < 1$ satisfies 
$1/p=(1-\theta)/p_0+\theta/p_1$. 
By similar manner, we define $L^{p,q}_{\sigma}(\R^n)$ 
for $1< p < \infty$ and $1\leq q \leq \infty$. 
Here we remark that
for $1<p<\infty$ and $1\leq q <\infty$
$C_{0,\sigma}^\infty(\R^n)$ is dense in $L^{p,q}_{\sigma}(\R^n)$ and 
the dual space of $L^{p,q}_{\sigma}(\R^n)$ 
is $L^{p^\prime,q^\prime}_{\sigma}(\R^n)$ where $1/p+1/p^\prime=1$ and
$1/q+1/q^\prime=1$.
Especially, 
$\bigl(L^{\frac{n}{n-1},1}_{\sigma}(\R^n)\bigr)^*=L^{n,\infty}_\sigma(\R^n)$. 

By the real interpolation, we also obtain the Helmholtz decomposition
for $1<p<\infty$ and $1\leq q \leq \infty$,
\begin{equation*}
L^{p,q}(\R^n)=L^{p,q}_{\sigma}(\R^n) \oplus G^{p,q}(\R^n),
\qquad \text{(direct sum)},
\end{equation*}
where $G^{p,q}(\R^n)=\{\nabla \pi \in L^{p,q}(\R^n)\,;\, 
\pi\in L^{p}_{loc}(\R^n)\}$ is the associated direct sum decomposition.
See, Borchers and Miyakawa \cite[Theorem 5.2]{Borchers Miyakawa AM1995}, 
Simader and Sohr \cite{Simader Sohr}, Yamazaki \cite{Yamazaki MA2000}, and
Geissert and Hieber and Nguyen \cite{Geissert Hieber Nguyen ARMA2016}.
By the virtue of the Helmholtz decomposition, 
for $1<p<\infty$ and $1\leq q <\infty$
it holds that
$(u,\PP v)=(\PP u,v)=(\PP u, \PP v)$ for all $u \in L^{p,q}(\R^n)$
and $v\in L^{p^\prime,q^\prime}(\R^n)$ with $1/p+1/p^\prime=1$ 
and $1/q+1/q^\prime=1$.

For $1<r<\infty$, 
we introduce $A_r=-\PP \Delta$ be the Stokes operator on $L_\sigma^r(\R^n)$ 
with the domain ${D}(A_r)=W^{2,r}(\R^n)\cap L^r_{\sigma}(\R^n)$.
We abbreviate $A_r$ to $A$ if we have no possibility of confusion. 
Then it is well known that $-A$ generates the bounded analytic 
$C_0$-semigroup $\{e^{-tA}\}_{t\geq 0}$ on $L^r_{\sigma}(\R^n)$.
Moreover, we note that adjoint of $A_r$ is $A_{r^\prime}$ and
 $(e^{-tA_r})^*=e^{-tA_{r^\prime}}$ where $1/r+1/r^\prime=1$.
With the fractional power $A^\alpha$, $0\leq \alpha \leq 1$
we see  that
\begin{equation*}
\|A^{\alpha} e^{-tA} a \|_r \leq Ct^{-\alpha} \|a\|_r.
\end{equation*}
by the standard semigroup theory. 
Moreover in case of $\R^n$, the Stokes 
semigroup is essentially the heat semigroup and has the kernel function
$(4\pi t)^{-\frac{n}{2}}\exp(-\frac{|x|^2}{4t})$. Hence by the Young inequality,
we easily obtain 
the following $L^p$-$L^q$ estimate.

\begin{proposition}\label{prop;LpLq}
Let $1<p \leq r <\infty$. Suppose $a \in L^p(\R^n)$. 
Then it holds that
\begin{equation}\label{eq;Lp Lq}
\begin{split}
&\|e^{t\Delta}a\|_r
\leq Ct^{-\frac{n}{2}(\frac{1}{p}-\frac{1}{r})}
\|a\|_p, 
\\
&\|\nabla e^{t\Delta}a\|_r\leq 
C t^{-\frac{n}{2}(\frac{1}{p}-\frac{1}{r})-\frac{1}{2}}
\|a\|_p,
\end{split}
\end{equation}
for all $t>0$, where $C=C(p,r)$ 
is independent of $t$ and $a$.
Moreover, suppose $a \in L^{p,\infty}(\R^n)$. 
Then it holds that
\begin{equation}\label{eq;wLpwLq}
\begin{split}
&\|e^{t\Delta}a\|_{r,\infty}
\leq Ct^{-\frac{n}{2}(\frac{1}{p}-\frac{1}{r})}
\|a\|_{p,\infty}, 
\\
&\|\nabla e^{t\Delta}a\|_{r,\infty}\leq 
C t^{-\frac{n}{2}(\frac{1}{p}-\frac{1}{r})-\frac{1}{2}}
\|a\|_{p,\infty},
\end{split}
\end{equation}
for all $t>0$.
\end{proposition}

\begin{remark}\rm
When we restrict 
$a$ within $L^p_{\sigma}(\R^n)$ and $L_{\sigma}^{p,\infty}(\R^n)$
we see that 
$e^{t\Delta} a \in 
L^r_{\sigma}(\R^n)$ and
$e^{t\Delta} a \in 
L_{\sigma}^{r,\infty}(\R^n)$, respectively, 
and that \eqref{eq;Lp Lq} 
and \eqref{eq;wLpwLq} also hold.
\end{remark}

For $1<p<\infty$ and $1\leq q \leq \infty$, let $A_{p,q}=-\PP \Delta$ on 
$L^{p,q}_{\sigma}(\R^n)$ with its domain
${D}(A_{p,q})
=\{u\in L_{\sigma}^{p,q}(\R^n)\,;\, \pt_j\pt_k u \in L^{p,q}(\R^n),\,\,
j,k=1,\dots,n\}$. Then $A$ is a closed operator on $L^{p,q}_{\sigma}(\R^n)$
and $(Au,v)=(u,Av)$ holds for 
$u \in {D}(A_{p,q})$ and $v\in {D}(A_{p^\prime,q^\prime})$,
$1/p+1/p^\prime=1$ and $1/q+1/q^\prime=1$.
Here, we note that $\{e^{-tA}\}_{t\geq 0}$
is  bounded and analytic on $L^{p,\infty}_{\sigma}(\R^n)$ but not
strongly continuous on $L^{p, \infty}_{\sigma}(\R^n)$.
On the other hand, in case of $1\leq q <\infty$
it is easy to see that by the real interpolation that
$-A$ generates a bounded analytic $C_0$-semigroup $\{e^{-tA}\}_{t\geq 0}$
on $L^{p,q}_{\sigma}(\R^n)$ with the estimate 
\begin{equation*}
\begin{split}
&\|e^{-tA} a\|_{r,q} 
\leq C t^{-\frac{n}{2}(\frac{1}{p}-\frac{1}{r})}
\|a \|_{p,q},
\\
&\|\nabla e^{-tA} a\|_{r,q} 
\leq C t^{-\frac{n}{2}(\frac{1}{r}-\frac{1}{r})-\frac{1}{2}}
\|a \|_{p,q},
\end{split}
\end{equation*}
for  $a \in L^{p,q}_{\sigma}(\R^n)$ and for $1<p<r<\infty$.

\subsection{critical estimate}
In order to construct mild solutions of (N-S)
we deal with
\begin{equation*}
\int_0^t e^{(t-s)\Delta}\PP f(s)\,ds=
\int_0^\infty e^{s\Delta} \PP f(t-s)\chi_{[0,t]}(s)\,ds, \qquad t>0.
\end{equation*}
Here, $\chi_A$ is 
the usual characteristic function on the set $A$,
i.e., $\chi_A(x)=1$ if $x\in A$, otherwise $\chi_A(x)=0$.
For this aim, we introduced the following lemmas in the previous work \cite{Okabe Tsutsui periodic},
besed on the real interpolation approach by
Meyer \cite{Meyer} and Yamazaki \cite{Yamazaki MA2000}. 
\begin{lemma}[{\cite[Lemma 4.1]{Okabe Tsutsui periodic}}] \label{modi}
Let $n \ge 3$ and $1\leq p <\frac{n}{2}$, and define $p<q<\infty$ with $\frac{1}{p} - \frac{1}{q} = \frac{2}{n}$.
Then it holds
\begin{equation*}
\left\| 
\int_0^\infty \mathbb{P} e^{s \Delta} g(s)\, ds 
\right\|_{q, \infty} 
\leq A_p 
\begin{cases}
\displaystyle \sup_{s>0} \|g(s)\|_{p, \infty}, &\text{if } p>1, \\
\displaystyle \sup_{s>0} \|g(s)\|_1, & \text{if }p=1.
\end{cases}
\end{equation*}
\end{lemma}
\begin{remark}\rm
(i) If $g \in L^{p,\infty}(\R^n)$ for some $1<p<\infty$, it is easy to
see that $\PP e^{s\Delta} g(s)=e^{s\Delta} \PP g(s)$ for a.e. $x\in \R^n$.

(ii) The bound $\|\mathbb{P} e^{s \Delta}g(s)\|_{q,\infty} \leq c s^{-1} \|g(s)\|_{p,\infty}$ is not enough for the convergence of the integral at both $s=0$ and $s=\infty$.
\end{remark}

We also apply Meyer's estimate \cite{Meyer} for the non-linear term.
See also, Yamazaki \cite{Yamazaki MA2000}. 

\begin{lemma}[\cite{Meyer},\, \cite{Yamazaki MA2000},\, {\cite[Lemma 4.2]{Okabe Tsutsui periodic}}] \label{Meyer}
Let $n \ge 2$ and $1 \le p < n$.
Denote $\frac{n}{n-1}\leq q <\infty$ with $\frac{1}{p} - \frac{1}{q} = \frac{1}{n}$.
then
\begin{equation*}
\left\|\int_0^\infty \nabla \mathbb{P} e^{s \Delta} g(s) \,ds \right\|_{q,\infty} 
\leq B_p
\begin{cases}
\; \displaystyle \sup_{s>0} \|g(s)\|_{p,\infty} &\text{if } p >1, \\
\; \displaystyle  \sup_{s>0} \|g(s)\|_1 & \text{if } p=1.
\end{cases}
\end{equation*}
\end{lemma}
As an application, Lemma \ref{modi} and 
Lemma \ref{Meyer} yields the continuity of the Duhamel terms associated with the forces and 
the nonlinear term.
\begin{lemma}\label{lem;contiduhamel}
Let $n\geq 3$ and $1<p<\frac{n}{2}$. 
For $f \in BC\bigl((0,\infty)\,;\,L^{p,\infty}(\R^n)\bigr)$
 it holds that 
 \begin{equation*}
 \int_0^t e^{(t-s)\Delta}\PP f(s)\,ds
 \in
 BC\bigl( (0,\infty)\,;\,L^{q,\infty}(\R^n)\bigr)
 \quad
 \text{with }
 \frac{1}{q}=\frac{1}{p}-\frac{2}{n}.
 \end{equation*}
\end{lemma}
\begin{remark}\rm\label{rem;contiduhamel}
In case $p=1$, 
Lemma \ref{lem;contiduhamel} is also valid 
with  a slight modification
as $\int_0^t \PP e^{(t-s)\Delta}f(s)\,ds$ for 
$f \in BC\bigl((0,\infty)\,;\,L^1(\R^n)\bigr)
$. 
Moreover, 
for $\int_0^t \nabla e^{(t-s)\Delta}\PP f(s)\,ds$
we similarly obtain the continuity.
\end{remark}
\begin{proof}[Proof of Lemma \ref{lem;contiduhamel}] 
Put $F(t)=\int_0^t e^{(t-s)\Delta}\PP f(s)\,ds$ for $t>0$.
We shall prove the continuity of $F$
at fixed $t>0$.
Taking $t_0<t$ and decomposing $t=\tau+t_0$, 
we remark that 
\begin{equation*}
F(t)=F(\tau + t_0)=
e^{\tau \Delta}F(t_0) + \int_0^\tau
e^{(\tau-s)\Delta}\PP f(s+t_0)\,ds
=e^{\tau \Delta}F(t_0) + \int_0^\tau
e^{s\Delta}\PP f(t-s+t_0)\,ds.
\end{equation*}
For sufficiently small $\ep>0$ we see that
\begin{multline*}
F(t)-F(t-\ep)=
[e^{\tau \Delta}-e^{(\tau-\ep)\Delta}]F(t_0)
\\+
\int_{\tau-\ep}^{\tau}
e^{s\Delta}\PP f(\tau-s+t_0)\,ds
+
\int_0^{\tau-\ep}e^{s\Delta}
\bigl[\PP f(\tau-s+t_0)-\PP f(\tau-\ep-s+t_0)\bigr]\,ds
\\=:F_1+F_2 + F_3.
\end{multline*}
So it is easy to see that 
$\lim\limits_{\ep\searrow 0}\|F_1\|_{q,\infty}=0$.
Next we estimate $F_2$ as follows
\begin{equation*}
\begin{split}
\|F_2\|_{q,\infty}
&=
\left\|
\int_{\tau-\ep}^{\tau}
e^{s\Delta}\PP f(\tau-s+t_0)\,ds
\right\|_{q,\infty}
\\
&\leq
C\int_{\tau-\ep}^{\tau}
\frac{1}{s}\, \|\PP f(\tau-s+t_0)\|_{p,\infty}\,ds
\\
&\leq C \log \Bigl( \frac{\tau}{\tau-\ep}\Bigr)
\sup_{0<s<\infty}\|\PP f(s)\|_{p,\infty}
\to 0 \quad\text{as }\ep \searrow 0.
\end{split}
\end{equation*}
Finally, we estimate $F_3$. By Lemma \ref{modi}, we have
\begin{equation*}
\begin{split}
\|F_3\|_{q,\infty}
&=\left\|
\int_0^{\infty}
e^{s\Delta}\chi_{[0,\tau-\ep]}(s)
\bigl[\PP f(\tau-s+t_0)- 
\PP f(\tau-\ep-s+t_0)\bigr]\,ds
\right\|_{q,\infty}
\\
&\leq A_p\sup_{0<s<\tau-\ep} 
\| \PP f(\tau-s+t_0) 
-\PP f(\tau-\ep -s +t_0)\|_{p,\infty}
\\
&=A_p \sup_{t_0<\sigma<\tau+t_0-\ep}
\|\PP f(\sigma+\ep)-\PP f(\sigma)\|_{p,\infty}
=
A_p \sup_{t_0<\sigma<t-\ep}
\|\PP f(\sigma+\ep)-\PP f(\sigma)\|_{p,\infty}.
\end{split}
\end{equation*}
Hence, the uniform continuity of 
$\PP f$ on $[t_0,t]$ yields the left continuity of
$F(t)$ at $t>0$.
The right continuity is derived 
by the same manner.
This ends the proof.
\end{proof}

The following lemmas play an important role for the local existence and for the uniqueness criterion
of weak mild solutions of (N-S).
For this aim, we recall 
the space $\widetilde{L}^{p,\infty}(\R^n)
:=\overline{L^{p,\infty}(\R^n)\cap L^\infty(\R^n)}^{\|\cdot\|_{p,\infty}}$ for $1<p<\infty$
and introduce a space $Y^{p,\infty}_{p^\prime}
=\{f \in L^{p,\infty}(\R^n)\,;\, 
f\in L^{p^\prime}(\R^n)\}$ for $1<p<p^\prime<\infty$.
\begin{lemma}\label{lem;density}
Let $1<p<p^\prime<\infty$. For every $\ep>0$
and $f \in BC\bigl([0,\infty)\,;\,
\widetilde{L}^{p,\infty}(\R^n)\bigr)$
there exists $f_\ep \in 
BC\bigl([0,\infty)\,;\, Y^{p,\infty}_{p^\prime}\bigr)$
such that 
\begin{equation*}
\sup_{0\leq s <\infty} \|f(s)-f_\ep(s)\|_{p,\infty}<\ep,
\end{equation*}
i.e., $BC\bigl([0,\infty)\,;\,Y^{p,\infty}_{p^\prime}\bigr)$ is a dense subspace within 
$BC\bigl([0,\infty)\,;\,
\widetilde{L}^{p,\infty}(\R^n)\bigr)$.
\end{lemma}
\begin{remark}\rm
For a finite interval $[0,T]$, 
we easily obtain the same density property.
Moreover, it is easy to see $BC\bigl([0,\infty)\,;\,
C_{0}^{\infty}(\R^3)\bigr)$ is a 
dense subspace within 
$BC\bigl([0,\infty)\,;\,L^1(\R^3)\bigr)$.
\end{remark}
\begin{proof}[proof of Lemma \ref{lem;density}]
To begin with, we note that $L^{p,\infty}(\R^n)=\overline{Y^{p,\infty}_{p^\prime}}^{\|\cdot\|_{p,\infty}}$.
Take $\ep>0$.
Since $f$ is uniformly continuous on $[n,n+1]$, there exists $N_n$ such that
if $|s-t|<\frac{1}{N_n}$
for $s,t \in [n,n+1]$ then
$\|f(s)-f(t)\|_{p,\infty}<\ep/5$, 
for each $n=0,1,2,\dots$.
So set $t_{n,k}=n+k/N_n$, $k=0,1,\dots,N_n$ and
take $\phi_{n,k}\in Y^{p,\infty}_{p^\prime}$
so that $\|u(t_{n,k})-\phi_{n,k}\|_{p,\infty}
<\ep/5$.
Then we construct a function $f_\ep$ as
\begin{equation*}
f_{\ep}(t):= \phi_{n,k}+
\frac{\phi_{n,k+1}-\phi_{n,k}}{1/N_n}(t-t_{n,k})
\qquad\text{if }
\quad
n+\frac{k}{N_n}\leq t \leq n+\frac{k+1}{N_n}.
\end{equation*}
Then we easily see that 
$f_{\ep} \in BC\bigl([0,\infty)\,;\,
Y^{p,\infty}_{p^\prime}\bigr)$
and that for $t_{n,k}\leq t \leq t_{n,k+1}$
it holds that $\|f(t)-f_{\ep}(t)\|_{p,\infty}<\ep$. This completes the proof.
\end{proof}
\begin{lemma}\label{lem;tildeL}
Let $n\geq 3$. 
Suppose $f \in 
BC\bigl([0,\infty)\,;\, 
\widetilde{L}^{\frac{n}{3},\infty}(\R^n)\bigr)$
for $n\geq 4$ and 
$f\in BC\bigl([0,\infty)\,;\,L^1(\R^3)\bigr)$.
Then it holds that
\begin{equation}\label{eq;X1}
\int_0^t \PP e^{(t-s)\Delta} f(s)\,ds \in
BC\bigl([0,\infty)\,;\, X^{n,\infty}_{\sigma}\bigr)
\quad\text{with}\quad
\lim_{t\to0} 
\left\|\int_0^t \PP e^{(t-s)\Delta} f(s)\,ds
\right\|_{n,\infty}=0.
\end{equation}
Similarly, for a tensor $g=\bigl(g_{jk}\bigr)_{j,k=1}^n$, $g \in 
BC\bigl([0,\infty)\,;\,
\widetilde{L}^{\frac{n}{2},\infty}(\R^n)\bigr)$ 
it holds that
\begin{equation*}
\int_0^t \nabla \cdot e^{(t-s)\Delta}\PP g(s)\,ds \in
BC\bigl([0,\infty)\,;\, X^{n,\infty}_{\sigma}\bigr)
\quad\text{with}\quad
\lim_{t\to0} 
\left\|\int_0^t \nabla \cdot e^{(t-s)\Delta}\PP g(s)\,ds
\right\|_{n,\infty}=0.
\end{equation*}
\end{lemma}
\begin{remark}\label{rem;DuhamelLocal}\rm
Lemma \ref{lem;tildeL} plays an crucial role to construct a weak mild solution 
$u\in BC\bigl([0,\infty)\,;\,X^{n,\infty}_{\sigma} \bigr)$ by the iteration scheme, where the uniqueness is guaranteed.
For $f\in BC\bigl([0,\infty)\,;\, 
L^{p,\infty}(\R^n)\bigr)$ with some 
$\frac{n}{3}<p\leq n$, 
we easily see that $\int_0^t e^{(t-s)\Delta}\PP f(s)\,ds \in BC\bigl([0,T)\,;\,
X_{\sigma}^{n,\infty}\bigr)$ with 
$\lim\limits_{t\to 0}
\bigl\|\int_0^t e^{(t-s)\Delta}\PP f(s)\,ds
\bigr\|_{n,\infty}=0$ for finite $T>0$
instead of \eqref{eq;X1}, estimating $F^1$,
$F^2$ and $F^3$ below 
just by $L^p$-$L^q$ estimate. 
\end{remark}
\begin{proof}\rm
Put $F(t)=\int_0^t e^{(t-s)\Delta}\PP f(s)\,ds$. 
We firstly show that $F(t) \to 0$ in $L^{n,\infty}(\R^n)$ 
as $t \to 0$.
Take $\eta>0$. By Lemma \ref{lem;density} 
choose
$f_\eta \in BC\bigl([0,\infty)\,;\,
Y^{\frac{n}{3},\infty}_{p}\bigr)$ with some
$\frac{n}{3}<p<\infty$ such that 
$\sup\limits_{0\leq s <\infty}\|f(s)-f_{\eta}(s)\|_{\frac{n}{3},\infty}<\frac{\eta}{2A_{\frac{n}{3}}}$, where $A_{\frac{n}{3}}$ is the constant in Lemma \ref{modi} when $p=\frac{n}{3}$.
By Lemma \ref{modi} we have
\begin{equation*}
\begin{split}
\|F(t)\|_{n,\infty} 
&\leq 
\Bigg\| \int_0^\infty  e^{(t-s)\Delta}
\PP 
[f(s)-f_{\eta}(s)]\chi_{[0,t]}(s)\,ds\Bigg\|_{n,\infty}
+
\Bigg\| \int_0^t e^{(t-s)\Delta}
\PP f_{\eta}(s)\,ds\Bigg\|_{n,\infty}
\\
&\leq
A_{\frac{n}{3}}\sup_{0\leq s <\infty}
\|f(s)-f_{\eta}(s)\|_{\frac{n}{3},\infty}
+
C\int_0^t (t-s)^{\frac{1}{2}-\frac{n}{2p}}
\|f_\eta(s)\|_{p}\,ds
\\
&\leq
\frac{\eta}{2} + Ct^{\frac{3}{2}-\frac{n}{2p}}
\sup_{0\leq s <\infty}\|f_\eta(s)\|_{\frac{n}{3},\infty}.
\end{split}
\end{equation*}
Since
$\frac{3}{2}-\frac{n}{2p}>0$, there exists $\delta>0$ such that if $0<t<\delta$ then
$\|F(t)\|_{n,\infty}<\eta$. This prove the
continuity of $F(t)$ at $t=0$. 

Next we show $F(t)\in X^{n,\infty}_{\sigma}$ for each $t>0$. It holds that for sufficiently small $\ep>0$
\begin{equation*}
\begin{split}
e^{\ep\Delta}F(t)-F(t)
&=
\int_0^t e^{(s+\ep)\Delta}\PP f(t-s)\,ds
-
\int_0^t e^{s\Delta}\PP f(t-s)\,ds
\\
&=\int_\ep^{t+\ep}e^{s\Delta}f(t+\ep-s)\,ds
-\int_0^t e^{s\Delta}\PP f(t-s)\,ds
\\
&=
\int_\ep^t e^{s\Delta}\PP[f(t+\ep-s)-f(t-s)]\,ds
+
\int_t^{t+\ep} e^{s\Delta}\PP f(t+\ep-s)\,ds
+
\int_0^{\ep} e^{s\Delta} f(t-s)\,ds
\\
&=:F^1+F^2+F^3.
\end{split}
\end{equation*}
By Lemma \ref{modi} with $p=\frac{n}{3}$, we have
\begin{equation*}
\|F^1\|_{n,\infty} \leq A_{\frac{n}{3}}
\sup_{\ep<s<t} \|f(t+\ep-s)-f(t-s)\|_{\frac{n}{3},\infty}
\leq 
A_{\frac{n}{3}}
\sup_{0<s<t-\ep}
\|f(s+\ep)-f(s)\|_{\frac{n}{3},\infty}.
\end{equation*}
Hence the uniform continuity of $f$ on 
$[0,t]$ yields $\|F^1\|_{n,\infty}\to 0$ as 
$\ep \to 0$.
Next we see that
\begin{equation*}
\|F^2\|_{n,\infty} \leq
C\int_t^{t+\ep}\frac{1}{s}\,
\sup_{0\leq s<\infty} 
\|f(s)\|_{\frac{n}{3},\infty}\,ds
\leq
C\log \left(\frac{t+\ep}{t}\right)
\sup_{0\leq s<\infty} 
\|f(s)\|_{\frac{n}{3},\infty}
\to 0 \quad\text{as }\ep \to 0.
\end{equation*}
Finally, we estimate $F^3$.
Take arbitrary $\eta>0$ and take $f_\eta$ as above. Then it holds that
\begin{equation*}
\begin{split}
\|F^3\|_{n,\infty} 
&\leq
\Bigg\|\int_0^{\ep} e^{s\Delta}
\PP[f(t-s)-f_\eta(t-s)]\,ds\Bigg\|_{n,\infty}
+
\Bigg\|\int_0^{\ep} e^{s\Delta}
\PP f_{\eta}(t-s)\,ds\Bigg\|_{n,\infty}
\\
&\leq \frac{\eta}{2}
+C\ep^{\frac{3}{2}-\frac{n}{2p}} 
\sup_{0\leq s<\infty}\|f_\eta(s)\|_p.
\end{split}
\end{equation*}
Then for sufficiently small $\ep>0$
we obtain $\|F^3\|_{n,\infty}<\eta$.
Therefore, $e^{\ep\Delta}F(t)-F(t)\to 0$ in 
$L^{n,\infty}(\R^n)$ as $\ep \to 0$.
By Theorem \ref{thm;maximal}, $F(t)\in X^{n,\infty}_{\sigma}$ for each $t>0$. 
Moreover, we note that 
$F(0)=0\in X^{n,\infty}_{\sigma}$.

For the case $f\in BC\bigl([0,\infty)\,;\,
L^1(\R^3)\bigr)$, take $f_\eta \in
BC\bigl([0,\infty)\,;\, C_{0,\sigma}^{\infty}(\R^3)\bigr)$ as above. Then same procedure above
holds true.

We remark that the same argument is applicable to $\int_0^t \nabla\cdot e^{(t-s)\Delta}\PP g(s)\,ds$.
This completes the proof.
\end{proof}
\subsection{Abstract evolution equations}
In this subsection, we develop
 a theory of abstract evolution equations
with the semigroup which is not strongly continuous at $t=0$, introduced by the previous work \cite{Okabe Tsutsui periodic}.

For a while, let $A$ be a general closed operator on a Banach space $X$ 
and $\{e^{tA}\}$ a bounded and analytic on $X$
with the estimates
\begin{equation}\label{eq;appendix 1}
\sup_{0<t<\infty}\|e^{t A} \|_{\mathcal{L}(X)} \leq N,
\quad
\|Ae^{tA}\|_{\mathcal{L}(X)} \leq \frac{M}{t}, 
\quad t>0,
\end{equation}
where $\mathcal{L}(X)$ is the space of 
all bounded linear operators on $X$ equipped with
the operator norm.
Especially, we note that $e^{tA}$ is strongly  continuous 
in $X$ for $t\neq 0$.

\begin{definition}
Let $\theta\in (0,1]$. We call $f$ is the H\"{o}lder continuous on $[0,\infty)$ with value in $X$ with the order $\theta$, if 
for every $T>0$ there exists $K_T>0$ such that 
\begin{equation*}
\| f(t)-f(s)\|_X \leq K_T|t-s|^\theta,
\qquad 0\leq t\leq T, \, 0 \leq s \leq T. 
\end{equation*}
\end{definition}

\begin{assumption}
Let $f: [0,\infty)\to X$. We assume for every $t>0$
\begin{equation*}\tag{A}
\lim_{\ep\searrow 0} \|e^{\ep A}f(t) -f(t)\|_X=0.
\end{equation*}
\end{assumption}

\begin{lemma}[{\cite[Lemma 3.1]{Okabe Tsutsui periodic}}]\label{thm;abstract evo}
Let $a \in X$ and let $f\in C\bigl([0,\infty)\,;\,X\bigr)$
 be the H\"{o}lder continous on $[0,\infty)$ with 
value in $X$ with order $\theta>0$ and satisfy Assumption.
Then
\begin{equation*}
 u(t)=e^{tA}a + \int_0^t e^{(t-s)A} f(s)\,ds
\end{equation*}
satisfies 
 $u \in C^1\bigl((0,\infty)\,;\,X\bigl)$, 
 $Au \in C\bigl((0,\infty)\,;\,X\bigr)$
 and
\begin{equation*}
\frac{d}{dt}u =Au + f \quad\text{in }X\quad t>0. \\
\end{equation*}
\end{lemma}
\begin{remark}\rm
We note that we need a restriction only on the external force $f$ not on initial data $a$. Moreover, Lemma \ref{thm;abstract evo} gives
no information on
the verification of the initial condition.
While, property of the adjoint operator $A^*$ and
the dual space $X^*$ has the possibility to 
verify of the initial condition with a suitable sense.
\end{remark}
\section{Global mild solution}
%
In this section, we give a proof of Theorem \ref{thm;global} except for (iv) and (vii).
For the remining part, see Section \ref{sec;p2.1}.

We firstly consider the case $n\geq 4$.
Define a successive appoximation $\{u_m\}_{m=1}^\infty$ by
\begin{equation}\label{eq;4dim}
\left\{\begin{split}
u_0(t)
&=e^{t\Delta}a+\int_{0}^t e^{(t-s)\Delta} \PP f(s)\,ds
=e^{t\Delta}a +
\int_0^{\infty} e^{s\Delta} \PP \widetilde{f}(s)\,ds,
\\
u_{m+1}(t)
&=u_0(t)+ G[u_m,u_m](t), \quad m=0,1,2,\dots,
\end{split}\right.
\end{equation}
for $t>0$, where $\widetilde{f}(s)=f(t-s)\chi_{[0,t]}(s)$ and 
\begin{equation}\label{eq;nonlinear4dim}
G[u,v](t)
=-\int_{0}^t 
\nabla \cdot e^{(t-s)\Delta}
\PP[u\otimes v](s)\,ds
=-\int_{0}^{\infty} \nabla \cdot e^{s\Delta} 
\PP[\widetilde{u}\otimes\widetilde{v}](s)\,ds.
\end{equation}
For the continuity of $\{u_m\}$, 
see Lemma \ref{lem;contiduhamel}.
Put 
$\mathcal{K}_m=\sup\limits_{0<t<\infty}\|u_m(t)\|_{n,\infty}$ for $m=0,1,2,\dots$. Then by Lemma \ref{modi} with $p=\frac{n}{3}$ and $q=n$,
we have
\begin{equation*}
\|u_0(t)\|_{n,\infty}
\leq C\|a\|_{n,\infty}+A_{\frac{n}{3}}
\sup_{0<s<\infty}\|\PP f(s)\|_{\frac{n}{3},\infty}
<\infty
\end{equation*}
for all $t>0$. 
Hence 
$\mathcal{K}_0\leq C\bigl(\|a\|_{n,\infty}
+\sup\limits_{0<s<\infty}\|\PP f(s)\|_{\frac{n}{3},\infty}\bigr)$. 
Moreover, by Lemma \ref{Meyer} with $p=\frac{n}{2}$ and $p=n$, we see that
\begin{equation*}
\begin{split}
\|G[u_m,u_m](t)\|_{n,\infty}
&\leq
B_{\frac{n}{2}} \sup_{0<s<\infty}\|\PP[\widetilde{u}_m\otimes \widetilde{u}_m](s)\|_{\frac{n}{2},\infty}
\\
&\leq 
C_1 \sup_{0<s<\infty}\|u_m(s)\|_{n,\infty}^2.
\end{split}
\end{equation*}
for all $t>0$. Hence we have
\begin{equation*}
\mathcal{K}_{m+1}\leq 
\mathcal{K}_0+ C_1\mathcal{K}_m^2,
\qquad m=0,1,2,\dots.
\end{equation*}
So if 
\begin{equation}\label{eq;small k}
\mathcal{K}_0<
\frac{1}{4C_1},
\end{equation}
then we obtain a uniform bound of 
$\{\mathcal{K}_m\}_{m=1}^\infty$ as
\begin{equation*}
\mathcal{K}_m \leq \frac{1-\sqrt{1-4C_1\mathcal{K}_0}}{2C_1}\leq 2\mathcal{K}_0,
\qquad m=0,1,2,\dots.
\end{equation*}
This yields a limit 
$u\in BC\bigl((0,\infty)\,;\, L_{\sigma}^{n,\infty}(\R^n)\bigr)$ with 
$\sup\limits_{0<t<\infty}\|u(t)\|_{n,\infty}
\leq 2\mathcal{K}_0$.
Then $u$ is a desired solution of (IE$^{*})$,
i.e., a weak mild solution of (N-S). 
To obtain a weak mild solution,
it suffices that 
$\| a \|_{n,\infty}+
\sup\limits_{0<t<\infty}
\|\PP f(t)\|_{\frac{n}{3},\infty}$ is small enough 
so that \eqref{eq;small k} holds.
This proves (i) of Theorem \ref{thm;global}.

Next, we additionally assume 
$\nabla a \in L^{\frac{n}{2},\infty}(\R^n)$ and modify $G[\cdot,\cdot]$ as 
\begin{equation*}
G^*[u,v](t)=
-\int_0^t e^{(t-s)\Delta}
\PP [u\cdot\nabla v](s)\,ds
=-\int_0^\infty
 e^{s\Delta}
\PP [\widetilde{u}\cdot\nabla\widetilde{v}](s)\,ds,
\quad t>0.
\end{equation*}
Here, we note that if 
$u,v\in L_{\sigma}^{n,\infty}(\R^n)\cap L^r(\R^n)$
with some $r>n$ and 
$\nabla u,\nabla v\in L^{q,\infty}(\R^n)$ with some $q \geq \frac{n}{2}$
, then we easily see that 
$\nabla \cdot e^{s\Delta}\PP [u\otimes v]
= e^{s\Delta}\PP [u\cdot\nabla v]$ 
for almost every $x \in \R^n$. 
Hence, if a mild solution 
$u$ has $L^r$-regularity for some $r>n$, 
we see that $u$ corresponds 
with a weak mild solution obtained above.

Consider the following successive approximation 
$\{u_m\}_{m=1}^\infty$ by
\begin{equation}\label{eq;modimodi}
\left\{\begin{split}
u_0(t)
&=e^{t\Delta}a+\int_{0}^t e^{(t-s)\Delta} \PP f(s)\,ds
=e^{t\Delta}a +
\int_0^{\infty} e^{s\Delta} \PP \widetilde{f}(s)\,ds,
\\
u_{m+1}(t)
&=u_0(t)+ G^*[u_m,u_m](t), \quad m=0,1,2,\dots.
\end{split}\right.
\end{equation}
Put 
$\mathcal{K}^*_m=\max\bigl\{
\sup\limits_{0<t<\infty}\|u_m(t)\|_{n,\infty},
\sup\limits_{0<t<\infty}
\|\nabla u_m(t)\|_{\frac{n}{2},\infty}
\bigr\}$ for $m=0,1,2,\dots$.
Note that $\|a\|_{n,\infty}\leq C\|\nabla u\|_{\frac{n}{2},\infty}$ by the Sobolev inequality. 
By Lemma \ref{modi} and Lemma \ref{Meyer}, we have
\begin{equation*}
\begin{split}
\|u_0(t)\|_{n,\infty} 
&\leq C\|\nabla a \|_{\frac{n}{2},\infty}
+
A_{\frac{n}{3}}\sup_{0<s<\infty}\|\PP f(s)\|_{\frac{n}{3},\infty},
\\
\|\nabla u_0(t)\|_{\frac{n}{2},\infty}
&\leq
C\|\nabla a\|_{\frac{n}{2},\infty} 
+ B_{\frac{n}{3}}
\sup_{0<s<\infty} 
\|\PP f(s)\|_{\frac{n}{3},\infty}.
\end{split}
\end{equation*}
Hence $\mathcal{K}_0^*\leq C\bigl(
\|\nabla a\|_{\frac{n}{2},\infty}
+
\sup\limits_{0<t<\infty}
\|\PP f(t)\|_{\frac{n}{3},\infty}
\bigr)$. 
Moreover, we have
\begin{equation*}
\begin{split}
\|G^*[u_m,u_m](t)\|_{n,\infty}
&\leq
A_{\frac{n}{3}}
\sup_{0<s<\infty}\|\PP [\widetilde{u}_m
\cdot\nabla \widetilde{u}_m](s)\|_{\frac{n}{3},\infty} 
\\
&\leq
CA_{\frac{n}{3}}
\sup_{0<s<\infty}\|u_m(s)\|_{n,\infty}
\sup_{0<s<\infty}\|\nabla u_m(s)\|_{\frac{n}{2},\infty},
\end{split}
\end{equation*}
and
\begin{equation*}
\begin{split}
\|\nabla G^*[u_m,u_m](t)\|_{\frac{n}{2},\infty}
&\leq
B_{\frac{n}{3}}
\sup_{0<s<\infty}\|\PP [
\widetilde{u}_m\cdot\nabla 
\widetilde{u}_m](s)\|_{\frac{n}{3},\infty} 
\\
&\leq
CB_{\frac{n}{3}}
\sup_{0<s<\infty}\|u_m(s)\|_{n,\infty}
\sup_{0<s<\infty}\|\nabla u_m(s)\|_{\frac{n}{2},\infty}
\end{split}
\end{equation*}
for all $t>0$.
Hence we obtain
\begin{equation*}
\mathcal{K}_{m+1}^* \leq \mathcal{K}_0^*
+ C_2 (\mathcal{K}_m^*)^2, \qquad m=0,1,2,\dots.
\end{equation*}
Therefore if 
\begin{equation}\label{eq;small mild}
\mathcal{K}_0^*<\frac{1}{4C_2},
\end{equation}
Then similarly we obtain a uniform bound 
$\mathcal{K}_m^*\leq 2 \mathcal{K}_0^*$ for 
$m=0,1,2,\dots$ and hence obtain
a limit $u \in BC\bigl((0,\infty)\,;\,
L_{\sigma}^{n,\infty}(\R^n)\bigr)$ with
$\nabla u \in BC\bigl((0,\infty)\,;\,L^{\frac{n}{2},\infty}(\R^n)\bigr)$, 
which is a desired mild solution of (N-S).
To obtain a mild solution, 
it is suffices that 
$\|\nabla a\|_{\frac{n}{2},\infty}+
\sup\limits_{0<t<\infty}
\|\PP f(t)\|_{\frac{n}{3},\infty}$ 
is small enough so that
\eqref{eq;small mild} holds. 
This proves (ii) of Theorem \ref{thm;global}.

Here for $\frac{n}{3}<p<\frac{n}{2}$, 
we also assume
$\nabla a \in L^{q,\infty}$ with 
$\frac{1}{q}=\frac{1}{p}-\frac{1}{n}$ and
assume $f\in BC\bigl([0,\infty)\,;\,
L^{p,\infty}(\R^n)\bigr)$.
We consider the successive approximation \eqref{eq;modimodi} again. 

Let $\frac{1}{r}
=\frac{1}{q}-\frac{1}{n}=\frac{1}{p}-\frac{2}{n}$.  
Put 
$\mathcal{L}_m
=\sup\limits_{0<t<\infty}{\|u_m(t)\|_{r,\infty}}$ for $m=0,1,2,\dots$.
Then we note that 
$\|a \|_{r,\infty}\leq C\|\nabla a\|_{q,\infty}$
by the Sobolev inequality.
Then by Lemma \ref{modi} we have
\begin{equation*}
\|u_0(t)\|_{r,\infty} \leq
C\|a\|_{r,\infty} + 
A_p \sup_{0<s<\infty} \|\PP f(s)\|_{p,\infty},
\end{equation*}
for all $t>0$.
Hence, 
$\mathcal{L}_0\leq 
C\bigl(\|a\|_{r,\infty}
+ \sup\limits_{0<t<\infty}
\| \PP f(t)\|_{p,\infty}\bigr)$.
By Lemma \ref{modi}, we have
\begin{equation*}
\begin{split}
\|G^*[u_m,u_m](t)\|_{r,\infty}
&\leq A_p \sup_{0<s<\infty}\|\PP [
\widetilde{u}_m\cdot\nabla 
\widetilde{u}_m](s)\|_{p,\infty}
\\
&\leq
CA_p \sup_{0<s<\infty}
\| u_m(s)\|_{r,\infty}
\sup_{0<s<\infty} 
\|\nabla u_m(s)\|_{\frac{n}{2},\infty},
\end{split}
\end{equation*}
for all $t>0$. Hence we have a linear recurrence 
\begin{equation*}
\mathcal{L}_{m+1} \leq \mathcal{L}_0 + C_3 \mathcal{K}_0^*\mathcal{L}_m,
\qquad m=0,1,2,\dots.
\end{equation*}
Therefore, if
\begin{equation}\label{eq;small l}
\mathcal{K}_0^* < \frac{1}{C_3},
\end{equation}
then we obtain a uniform bound 
\begin{equation*}
\mathcal{L}_m \leq \frac{\mathcal{L}_0}{1-C_3\mathcal{K}_0^*}=:\mathcal{L}_*,
\qquad m=0,1,2,\dots,
\end{equation*}
which yields a limit 
$u \in BC\bigl((0,\infty)\,;
\,L_{\sigma}^{r,\infty}(\R^n)\bigr)$ with 
$\sup\limits_{0<t<\infty}
\|u(t)\|_{r,\infty}\leq \mathcal{L}_*$.

Similarly, put 
$\mathcal{M}_m=\sup\limits_{0<t<\infty}
\|\nabla u(t)\|_{q,\infty}$ for $m=0,1,2,\dots$.
By Lemma \ref{Meyer}, we see that
\begin{equation*}
\| \nabla u_0(t)\|_{q,\infty}
\leq
C\|\nabla a\|_{q,\infty}+B_p\sup_{0<s<\infty}
\|\PP f(s)\|_{p,\infty},
\end{equation*}
for all $t>0$. Hence, 
$\mathcal{M}_0\leq C\bigl(
\|\nabla a\|_{q,\infty}+
\sup\limits_{0<t<\infty}
\|\PP f(t)\|_{q,\infty}\bigr)$.
By Lemma \ref{Meyer}, we also have
\begin{equation*}
\begin{split}
\|\nabla G^*[u_m,u_m](t)\|_{q,\infty}
&\leq B_p \sup_{0<s<\infty}\|\PP[
\widetilde{u}_m\cdot\nabla 
\widetilde{u}_m](s)\|_{p,\infty}
\\
&\leq CB_p \sup_{0<s<\infty}\|u_m(s)\|_{n,\infty}
\sup_{0<s<\infty} \|\nabla u_m(s)\|_{q,\infty},
\end{split}
\end{equation*}
for all $t>0$. Hence we obtain
\begin{equation*}
\mathcal{M}_{m+1}\leq \mathcal{M}_0+
C_4\mathcal{K}_0^* \mathcal{M}_m,
\qquad m=0,1,2,\dots.
\end{equation*}
Therefore, if 
\begin{equation}\label{eq;small m}
\mathcal{K}_0^* <\frac{1}{C_4},
\end{equation}
then we obtain a uniform bound
\begin{equation*}
\mathcal{M}_m \leq 
\frac{\mathcal{M}_0}{1-C_4\mathcal{K}^*_0}=:\mathcal{M}_*,
\qquad m=0,1,2,\dots,
\end{equation*}
which yields a limit 
$\nabla u \in BC\bigl((0,\infty)\,;\,
L^{q,\infty}(\R^n)\bigr)$ with 
$\sup\limits_{0<t<\infty}
\|\nabla u (t)\|_{q,\infty} \leq \mathcal{M}_*$.
As conclusion of the proof of (iii) 
of Theorem \ref{thm;global},
the smallness condition for
$\|\nabla a\|_{\frac{n}{2},\infty}
+\sup\limits_{0<t<\infty}
\|\PP f(t)\|_{\frac{n}{3},\infty}$
is required so that \eqref{eq;small mild},
\eqref{eq;small l} and \eqref{eq;small m} hold.
Moreover, since $G^*[u_m,u_m](t)=G[u_m,u_m](t)$ 
for almost every $x\in \R^n$ provided 
$u_m(t)\in L_{\sigma}^{n,\infty}(\R^n)
\cap L^{r,\infty}(\R^n)$ for some $r>n$ and
$\nabla u_m(t)\in L^{q,\infty}(\R^n)$ for some
$q\geq \frac{n}{2}$, 
the mild solution
$u$ above is also a weak mild solution of (N-S).

Finally, we consider the case $n=3$.
In this situation, $\frac{n}{3}=1$. 
Due to the end point case in 
Lemmas \ref{modi} and \ref{Meyer}, 
we restrict the class of external forces
within $BC\bigl([0,\infty)\,;\,
L^{1}(\R^3)\bigr)$.
Then by the almost same procedure of the proof of (i) of Theorem \ref{thm;global} above,
we obtain a (modified) weak mild solution
of (N-S), i.e., a solution of (IE$^{**}$)
by the smallness of 
$\|a\|_{3,\infty}+
\sup\limits_{0<t<\infty}\|f(t)\|_{1}$,
 see Remark \ref{rem;weakmildsol}.
On the other hand, we unable to handle
$\sup\limits_{0<t<\infty}
\|\PP[u_m\cdot\nabla u_m](t)\|_{1}$ 
with our approach.
Hence our approach does not allow us to construct
a mild solution of (N-S), 
assuming only critical regularity for 
$\nabla a \in L^{\frac{3}{2},\infty}(\R^3)$ 
and $f\in BC\bigl([0,\infty)\,;\,L^1(\R^3)\bigr)$.

For $1<p<\frac{3}{2}$,
 we assume additional regularities for 
$\nabla a \in L^{\frac{3}{2},\infty}(\R^3)\cap
L^{q,\infty}(\R^3)$, 
$\frac{1}{q}=\frac{1}{p}-\frac{1}{3}$ and
$f\in BC\bigl([0,\infty)\,;\,
L^1(\R^3)\cap L^{p,\infty}(\R^3)\bigr)$.
Then noting that $G^{*}[u_m,u_m](t)=G[u_m,u_m](t)$
for almost every $x\in \R^3$, we are able to
construct a mild solution of (N-S),
with $\{\mathcal{K}_m\}$, $\{\mathcal{L}_m\}$
and $\{\mathcal{M}_m\}$ provided 
$\|\nabla a\|_{\frac{3}{2},\infty}
+\sup\limits_{0<t<\infty}\| \PP f(t)\|_1$ 
is small enough.
This yields the (v) and (iv) 
of Theorem \ref{thm;global}.
\section{Proof of Theorem \ref{thm;maximal}}
In this section, 
we prove Theorem \ref{thm;maximal}.
Let $f \in D(-\Delta)$.
Take $\ep>0$ and $\phi \in C_{0,\sigma}^\infty(\R^n)$.
Since $e^{t\Delta}f \rightharpoonup f$ weakly $*$
in $L^{n,\infty}_{\sigma}(\R^n)$ 
as $t\searrow 0$, 
for $\delta>0$ there exists $0< \eta=\eta(\ep,\delta,f,\phi)<\ep$ such that
\begin{equation*}
\bigl|(e^{\eta\Delta}f-f,\phi)\bigr|<\delta.
\end{equation*}
Then we see that
\begin{equation*}
\begin{split}
\bigl|(e^{\ep\Delta}f-f,\phi)\bigr|
&\leq
\bigl|(e^{\ep\Delta}f-e^{\eta\Delta}f,\phi)\bigr|+
\bigl|(e^{\eta\Delta}f-f,\phi)\bigr|\\
&\leq
\left(\int_\eta^{\ep}\Delta e^{\theta\Delta}f\,d\theta, \phi\right)+\delta
\\
&\leq
C\int_\eta^{\ep} \| (-\Delta)f\|_{n,\infty}
\|\phi\|_{\frac{n}{n-1},1}\,d\theta +\delta
\\
&\leq C(\ep-\eta)\|(-\Delta)f\|_{n,\infty}
\|\phi\|_{\frac{n}{n-1},1}+\delta
\\
&\leq
C\ep \|(-\Delta)f\|_{n,\infty} 
\|\phi\|_{\frac{n}{n-1},1}+\delta,
\end{split}
\end{equation*}
where $C=C\sup\limits_{0<\theta<\infty}
\|e^{\theta\Delta}\|_{\mathcal{L}(L^{n,\infty})}$
is independent of $\ep, \delta, \eta,f, \phi$.
Since $\delta>0$ is arbitrary, we obtain
\begin{equation*}
\bigl|(e^{\ep\Delta}f-f,\phi)\bigr|
\leq C\ep \|(-\Delta)f\|_{n,\infty}
\|\phi\|_{\frac{n}{n-1},1},
\end{equation*}
which implies $\|e^{\ep\Delta}f-f\|_{n,\infty}\to 0$ as $\ep \searrow 0$. By the density, 
it holds that $f \in \overline{D(-\Delta)}^{\|\cdot\|_{n,\infty}}$ satisfies (A).
The opposite is trivial. This completes the proof of Theorem \ref{thm;maximal}.
\section{Local strong solution}
In this section,
we construct a local in time 
weak mild and a mild solution of (N-S), 
according to Fujita and Kato 
\cite{Fujita Kato ARMA1964},
Kato \cite{Kato MZ1984},
Giga and Miyakawa \cite{Giga Miyakawa ARMA1985},
Kozono and Yamazaki \cite{Kozono Yamazaki MZ1998}.
Firstly, we construct local in time weak mild solution of
(N-S) in $X_{\sigma}^{n,\infty}$ and give 
a uniqueness criterion for the
weak mild solutions of (N-S) in the class 
$BC\bigl([0,T)\,;\,L_{\sigma}^{n,\infty}(\R^n)\bigr)$ with specific data which has less singularity.
Next we show the existence of local in time
mild solutions of (N-S), and derive its regularity.
Then we observe the constructed mild solution of (N-S) becomes a strong solution of (N-S).
Finally, we refer to the uniqueness criterion
introduced by Kozono and Yamazaki \cite{Kozono Yamazaki MZ1998}.
\subsection{Existence of local weak mild solutions; proof of (i) and (ii) in 
Theorem \ref{thm;localX}}
In this subsection,
we give the proof of (i) and (ii)
in Theorem \ref{thm;localX}. (For the proof
of (iii) of Theorem \ref{thm;localX} 
see Section \ref{sec;p2.3} below.) 
We firstly construct local
in time weak mild solutions of (N-S). 
However, since the force 
$f \in BC\bigl([0,\infty)\,;\,
\widetilde{L}^{\frac{n}{3},\infty}(\R^n)\bigr)$
for $n\geq 4$
or $f \in BC\bigl([0,\infty)\,;\,L^1(\R^3)\bigr)$
have just critical regularity, 
only $\|\cdot\|_{n,\infty}$ is available for the 
iteration method.
For this reason, we have to modify the usual 
Fujita-Kato approach.
Hereafter we assume 
$a\in X_{\sigma}^{n,\infty}=
\overline{D(-\Delta)}^{\|\cdot\|_{n,\infty}}$
and
$f \in BC\bigl([0,\infty)\,;\,
\widetilde{L}^{\frac{n}{3},\infty}(\R^n)\bigr)$
for the case $n\geq 4$, 
$f\in BC\bigl([0,\infty)\,;\,L^{1}(\R^3)\bigr)$.
For simplicity,
we introduce the notation $F(t):=\int_0^t \PP e^{(t-s)\Delta}f(s)\,ds$ and 
$G[u,v](t):=-\int_0^t \nabla \cdot e^{(t-s)\Delta}\PP [u\otimes v](s)\,ds$.

We reduce the following successive approximation:
\begin{equation*}
\left\{\begin{split}
&u_0(t)=e^{t\Delta}a +F(t), \\
&u_{n+1}(t)=u_0 +G[u_n,u_n](t), \quad n=0,1,2,\dots,
\end{split}\right.
\end{equation*}
to
\begin{equation*}
\left\{\begin{split}
&w_0(t)=e^{t\Delta}a -a+F(t), \\
&w_{n+1}(t)=w_0(t) +G[a,a](t)
+G[a,w_n](t)+G[w_n,a](t)+G[w_n,w_n](t), \quad n=0,1,2,\dots,
\end{split}\right.
\end{equation*}
putting $w_n(t)=u_n(t)-a$.

We note that if $g\in X_{\sigma}^{n,\infty}$, 
then $g\otimes g\in 
\widetilde{L}^{\frac{n}{2},\infty}(\R^n)$.
Hence Lemma \ref{lem;tildeL} yields
$w_n\in BC\bigl([0,\infty)\,;\,X_{\sigma}^{n,\infty}\bigr)$ with 
\begin{equation*}
\lim_{t\to 0}\|w_0(t)+G[a,a](t)\|_{n,\infty}=0.
\end{equation*}
Moreover, since $a\in X_{\sigma}^{n,\infty}$
for sufficiently small $\ep>0$ 
we choose
 $a_\ep \in L^r(\R^n)$ with some $r>n$ so that
 $\|a-a_{\ep}\|_{n,\infty}<\ep$.
Then we see that
\begin{equation*}
\begin{split}
\|G[a,w_n](t)\|_{n,\infty}
&\leq \|G[a-a_{\ep},w_n](t)\|_{n,\infty}+
\|G[a_{\ep},w_n](t)\|_{n,\infty}
\\
&\leq
B_{\frac{n}{2}}
\sup_{0<s<t}\|(a-a_\ep)
\otimes w_n\|_{\frac{n}{2},\infty}
+
C\int_0^t (t-s)^{-\frac{n}{2r}-\frac{1}{2}}
\|a_{\ep}\|_r \|w_n(s)\|_{n,\infty}\,ds
\\
&\leq
\bigl( CB_{\frac{n}{2}}\ep
+Ct^{\frac{1}{2}-\frac{n}{2r}}\|a_{\ep}\|_{r}\bigr)
\sup_{0<s<t}\|w_n(s)\|_{n,\infty}
\end{split}
\end{equation*}
Hence, with sufficient small $\ep>0$ and
 $T^\prime=T^\prime(a,\ep)>0$ we obtain
\begin{equation}\label{eq;ga1}
\sup_{0<s<t}\|G[a,w_n](s)\|_{n,\infty}+
\sup_{0<s<t}\|G[w_n,a](s)\|_{n,\infty}
<\frac{1}{2}\sup_{0<s<t}\|w_n(s)\|_{n,\infty}
\quad\text{for } 0<t<T^\prime.
\end{equation}
Then for $0<t<T^\prime$ we have
\begin{equation*}
\sup_{0<s<t}\|w_{n+1}(s)\|_{n,\infty}
\leq 
\sup_{0<s<t}\|w_{0}(s)+G[a,a](s)\|_{n,\infty}
+\frac{1}{2}\sup_{0<s<t}\|w_n(s)\|_{n,\infty}
+C_5 \left(\sup_{0<s<t}\|w_n(s)\|_{n,\infty}\right)^2.
\end{equation*}
So it suffices to take $T>0$ so that
\begin{equation*}
\sup_{0<s<T}\|w_0(s)+G[a,a](s)\|_{n,\infty}
\leq \frac{1}{16C_5}
\end{equation*}
in order to obtain a limit 
$u\in BC\bigl([0,T)\,;
\,X_{\sigma}^{n,\infty}\bigr)$
which is a desired weak mild solution of (N-S).

Finally, we note that 
Remark \ref{rem;DuhamelLocal} implies
the existence of weak mild solution of (N-S)
for $f \in BC\bigl([0,\infty)\,;\,
L^{p,\infty}(\R^n)\bigr)$ with some 
$\frac{n}{3}<p\leq n$ by the same manner above.
This completes the proof of (i) and (ii) in Theorem \ref{thm;localX}.
\subsection{Uniqueness of weak mild solutions;
proof of Theorem \ref{thm;local weak uniqueness}}
\label{subsec;UWMS}
In this subsection, we give a proof of 
Theorem \ref{thm;local weak uniqueness}.
We assume $a\in X_{\sigma}^{n,\infty}$
and $f\in BC\bigl([0,T)\,;\,
\widetilde{L}^{p,\infty}(\R^n)\bigr)$ for $n\geq 4$ and $f\in BC\bigl([0,T)\,;\,L^{1}(\R^3)\bigr)$.
Let $u,v \in BC\bigl([0,T)\,;\,
L_{\sigma}^{n,\infty}(\R^n)\bigr)$ be associated weak mild solutions of (N-S).
Moreover, let $w\in
BC\bigl([0,\tau_0]\,;\,X_{\sigma}^{n,\infty}\bigr)$ be a local weak mild solution
of (N-S) by Theorem \ref{thm;localX} with some $0<\tau_0<T$. 
We recall the notation $G[u,v](t)
=-\int_0^t\nabla\cdot e^{(t-s)\Delta}\PP[u\otimes v](s)\,ds$.
Then putting $U(t):=u(t)-w(t)$, we have
by Lemma \ref{Meyer} that
\begin{equation*}
\begin{split}
\|U(t)\|_{n,\infty}
\leq &
\|G[u-a,U](t)\|_{n,\infty}+
\|G[a,U](t)\|_{n,\infty}+
\|G[U,w-a](t)\|_{n,\infty}+
\|G[U,a](t)\|_{n,\infty}
\\
\leq &
CB_{\frac{n}{2}}
\bigl(\sup_{0<s<t}\|u(s)-a\|_{n,\infty}
+
\sup_{0<s<t}\|w(s)-a\|_{n,\infty}\bigr)
\sup_{0<s<t}\|U(s)\|_{n,\infty}
\\
&+\|G[a,U](t)\|_{n,\infty}+\|G[U,a](t)\|_{n,\infty}.
\end{split}
\end{equation*}
Since $a \in X_{\sigma}^{n,\infty}$, by the same argument as in \eqref{eq;ga1}, we can take 
sufficiently small $T^\prime=T^\prime(a)>0$ so that
\begin{equation*}
\sup_{0<s<t}\|G[a,U](s)\|_{n,\infty}+
\sup_{0<s<t}\|G[U,a](s)\|_{n,\infty}
< \frac{1}{4}\sup_{0<s<t}\|U(s)\|_{n,\infty},
\qquad\text{for } 0<t<T^\prime.
\end{equation*}
Moreover, $u$ and $w$ is continuous at $t=0$, 
there exists sufficiently small 
$T^{\prime\prime}>0$ such that
\begin{equation*}
\sup_{0<s<t}\|u(s)-a\|_{n,\infty}
+
\sup_{0<s<t}\|w(s)-a\|_{n,\infty}
<\frac{1}{4CB_{\frac{n}{2}}}
\qquad \text{for } 0<t<T^{\prime\prime}.
\end{equation*}
Therefore, we see 
\begin{equation*}
\sup_{0<s<t}\|U(s)\|_{n,\infty} \leq 
\frac{1}{2}
\sup_{0<s<t}\|U(s)\|_{n,\infty}, 
\qquad \text{for }
0<t<\min\{T^\prime,\,T^{\prime\prime}\}.
\end{equation*}
Hence we concede $u\equiv w\equiv v$ on $[0,\tau]$ with sufficiently small $\tau>0$.

Next we set $T_{\max}:=\sup\{\tau\,;\,
u\equiv v \text{ on }[0,\tau],\,\, u(\tau)\in X_{\sigma}^{n,\infty}\}$. If $T_{\max}=T$ then there is nothing to prove. Let us assume 
$T_{\max}<T$. By the continuation, we have $u(T_{\max})\in X_{\sigma}^{n,\infty}$.
So by the above argument again, 
the construction of a local weak mild solution of (N-S)
with the initial data $u(T_{\max})$ yields 
that $u\equiv v$ on $[0,T_{\max}+\ep]$ and 
$u(T_{\max}+\ep)\in X_{\sigma}^{n,\infty}$
with some $\ep>0$,
which contradicts the definition of $T_{\max}$.
This completes the proof of Theorem \ref{thm;local weak uniqueness}.
\subsection{Existence of local mild solutions}

\begin{theorem}\label{thm;local}
Let $a \in X_\sigma^{n,\infty}$ and let
$f\in BC\bigl([0,\infty)\,;\,
L^{n,\infty}(\R^n)\bigr)$.
Then there exist $T>0$ and a mild solution $v$ 
on $[0,T)$ of (N-S) which satisfies
\begin{alignat}{1}
\label{eq;v Ln}
v &\in BC\bigl([0,T)\,;\,
L_{\sigma}^{n,\infty}(\R^n)\bigr), \\
\label{eq;v rho}
t^{\frac{1}{2}-\frac{n}{2\rho}}v 
& \in BC\bigl([0,T)\,;\,
L_{\sigma}^{\rho}(\R^n)\bigr) \quad\text{for }n<\rho\leq \infty,\\
\label{eq;v nabla Ln}
t^{\frac{1}{2}}\nabla v 
&\in BC\bigl([0,T)\,;\,
L^{n,\infty}(\R^n)\bigr), \\
\label{eq;nablav rho}
t^{1-\frac{n}{2\varrho}}\nabla v
&\in BC\bigl([0,T)\,;\, L^{\varrho}(\R^n)\bigr)
\quad \text{for } n < \varrho <\infty.
\end{alignat}
Moreover, if $a \in X_{\sigma}^{n,\infty}\cap L^r(\R^n)$ for some $r>0$,
then there exists $\eta=\eta(n,r)>0$ such that
the existence time $T>0$ is estimated as
\begin{equation*}
T\geq
\min
\left\{
1, \Bigl(\frac{\eta}{\|a\|_r
+\sup\limits_{0<t<\infty}
\|\PP f(t)\|_{n,\infty}}\Bigr)^{\frac{2r}{r-n}}
\right\}.
\end{equation*}
\end{theorem}
\begin{remark}\label{rem;local}\rm
(i) Recalling the argument in previous subsections, 
we note that the mild solution $v$ 
obtained by Theorem \ref{thm;local} in the
class $v \in BC\bigl([0,T)\,;\,X_\sigma^{n,\infty}\bigr)$.

(ii) If $a \in L^{n,\infty}_{\sigma}(\R^n)\cap L^{r}(\R^n)$, instead of $a \in X_{\sigma}^{n,\infty}$ or $X_{\sigma}^{n,\infty}\cap L^r(\R^n)$,
then we obtain a mild solution $v$ satisfies
\eqref{eq;v Ln}, \eqref{eq;v rho}, \eqref{eq;v nabla Ln} and \eqref{eq;nablav rho}, except for $t=0$ and satisfies 
$v(t)\rightharpoonup a$ weakly $*$ in
$L^{n,\infty}_{\sigma}(\R^n)$ as $t\searrow 0$.
\end{remark}

\begin{proof}
We construct a local mild solution of (N-S) with the following successive approximation,
\begin{equation}\label{eq;iteration}
\left\{\begin{split}
& v_0(t)=e^{t\Delta} a + \int_0^t e^{(t-s)\Delta}\PP f(s)\,ds,\\
&v_{m+1}(t)=v_0(t)+G^*[v_m,v_m](t),
\end{split}\right.
\end{equation}
where $G^*[u,v](t)=-\int_0^t 
e^{(t-s)\Delta}\PP [u\cdot\nabla v](s)\,ds$.

Let $n<r<\infty$ and put
\begin{equation}
K_m=K_m(T)=\sup_{0<t<T}t^{\frac{1}{2}-\frac{n}{2r}}\|v_m(t)\|_r, \qquad j=0,1,2,\dots.
\end{equation}
Here we abbreviate $K_m(T)$ to $K_m$ if 
$0<T\leq \infty$ is obvious from the context.
We firstly derive the uniform bound for $\{K_m\}_{m=0}^\infty$. By $L^p$-$L^q$ estimate of the
Stokes semigroup, we see that
$\|e^{t\Delta}a\|_r\leq Ct^{-\frac{1}{2}+\frac{n}{2r}}\|a\|_{n,\infty}$ and
\begin{equation*}
\begin{split}
\left\| \int_0^t e^{(t-s)\Delta}\PP f(s)\,ds\right\|_r 
&\leq
C\int_0^t (t-s)^{-\frac{1}{2}+\frac{n}{2r}}
\sup_{0<s<\infty}\| \PP f(s)\|_{n,\infty}\,ds
\\
&\leq
C\sup_{0<s<\infty}\| \PP f(s)\|_{n,\infty}
t^{\frac{1}{2}+\frac{n}{2r}},
\end{split}
\end{equation*}
For all $t>0$. Hence we obtain that
\begin{equation*}
K_0(T)\leq C\|a\|_{n,\infty} + 
C\sup_{0<s<\infty}\| \PP f(s)\|_{n,\infty}T<\infty.
\end{equation*}
Next for every 
$\phi \in C_{0,\sigma}^\infty(\R^n)$
we have by integral by parts,
\begin{equation*}
\begin{split}
\bigl|\bigl(G^*[v_m,v_m](t),\phi\bigr)\bigr|
&\leq 
\int_0^t \bigl|\bigl(v_m(s)\cdot\nabla e^{(t-s)\Delta}\phi,v_m(s)\bigr)\bigr|\,ds
\\
&\leq
\int_0^t \|v_m(s)\|_r^2\|\nabla e^{(t-s)\Delta}\phi\|_{\frac{r}{r-2}}\,ds
\\
&\leq C\int_0^t
(t-s)^{-\frac{n}{2r}-\frac{1}{2}}\|v_m(s)\|_r\|\phi\|_{\frac{r}{r-1}}\,ds\\
&\leq
C_5K_m^2\|\phi\|_{\frac{r}{r-1}}t^{-\frac{1}{2}+\frac{n}{2r}},
\end{split}
\end{equation*}
for all $0<t<T$, where $C_6$ is independent of $t$ and $\phi$.
Then by the duality we obtain that
\begin{equation*}
\sup_{0<t<T} t^{\frac{1}{2}-\frac{n}{2r}}
\| G^*[v_m,v_m](t)\|_{r} \leq C_6 K_m, 
\end{equation*}
for all $m=0,1,2,\dots$, and $0<T<\infty$. Hence we have
\begin{equation*}
K_{m+1} \leq K_0 + C_6K_m^2, \qquad m=0,1,2,\dots.
\end{equation*}
Therefore if
\begin{equation}\label{eq;small K}
K_0 < \frac{1}{4C_6},
\end{equation}
then we have a uniform bound of $\{K_m\}$ as
\begin{equation*}
K_m \leq \frac{1-\sqrt{1-4C_6K_0}}{2C_6}\leq 2K_0,
\qquad m=0,1,2,\dots.
\end{equation*}
For a moment we assume \eqref{eq;small K} for some $T>0$.
The uniform bound yields a limit 
$v \in C\bigl((0,T)\,;\, L_{\sigma}^r(\R^n)\bigr)$
with
\begin{equation}\label{eq;v1}
\lim_{m\to\infty}\sup_{0<t<T}
t^{\frac{1}{2}-\frac{n}{2r}}\|v_m(t)-v(t)\|_r=0
\quad\text{and}\quad
\sup_{0<t<T}t^{\frac{1}{2}-\frac{n}{2r}}
\|v(t)\|_{r}\leq 2K_0.
\end{equation}

Next we put $L_m=L_m(T)=
\sup\limits_{0<t<T}\|v_m(t)\|_{n,\infty}$
for $m = 0,1,2,\dots$.
Similarly, since $\|e^{t\Delta}a\|_{n,\infty}\leq 
C\|a\|_{n,\infty}$ and
\begin{equation*}
\left\| \int_0^t e^{(t-s)\Delta}\PP f(s)\,ds
\right\|_{n,\infty}
\leq C\int_0^t \sup_{0<s<\infty}
\|\PP f(s)\|_{n,\infty}\,ds
\leq C\sup_{0<s<\infty}\|\PP f(s) \|_{n,\infty}t
\end{equation*}
for all $t>0$. 
Hence we have
\begin{equation*}
L_0(T) \leq \|a\|_{n,\infty} + \sup_{0<t<T}
\|\PP f(t)\|_{n,\infty}T<\infty.
\end{equation*}
For every $\phi 
\in C_{0,\sigma}^\infty(\R^n)$, 
we have by integral by parts that
\begin{equation*}
\begin{split}
\bigl|\bigl(G^*[v_m,v_m](t),\varphi\bigr)\bigr|
&\leq
\int_0^t \bigl|\bigl(
v_m(s)\cdot\nabla e^{(t-s)\Delta}\phi, 
v_m(s)\bigr)\bigr|,ds
\\
&\leq
C\int_0^t
\|v_m(s)v_m(s)\|_{\frac{rn}{n+r},\infty}
\|\nabla e^{(t-s)\Delta}\phi
\|_{\frac{rn}{rn-n-r},1}\,ds
\\
&\leq
C\int_0^t 
\|v_m(s)\|_{r}\|v_m(s)\|_{n,\infty}
(t-s)^{-\frac{n}{2r}-\frac{1}{2}}
\|\phi\|_{\frac{n}{n-1},1}\,ds \\
&\leq 
C_7 K_m L_m \|\phi\|_{\frac{n}{n-1},1},
\end{split}
\end{equation*}
for all $0<t<T$. Then since 
$L_{\sigma}^{n,\infty}(\R^n)
=\bigl(L_{\sigma}^{\frac{n}{n-1},1}(\R^n)\bigr)^{*}$
and since $C_{0,\sigma}^\infty(\R^n)$ is dense in 
$L_{\sigma}^{\frac{n}{n-1},1}(\R^n)$, the duality implies
\begin{equation*}
\sup_{0<t<T}\| G^*[v_m,v_m](t)\|_{n,\infty}
\leq C_7K_mL_m\leq2C_7K_0 L_m,\qquad
m=0,1,2,\dots.
\end{equation*}
Hence, we have a linear recurrence
\begin{equation*}
L_m \leq L_0 + 2C_7K_0L_m,\qquad m=0,1,2,\dots.
\end{equation*}
Therefore, if
\begin{equation}\label{eq;small L}
K_0<\frac{1}{2C_7},
\end{equation}
then we have a uniform bound
\begin{equation*}
L_m \leq \frac{L_0}{1-2C_7K_0}=:L_*, 
\qquad m=01,2,\dots,
\end{equation*}
which yields a limit 
$v\in C\bigl((0,T)\,;\, L^{n,\infty}\bigr)$ with
\begin{equation}\label{eq;v2}
\lim_{m\to\infty}\sup_{0<t<T}\|v_m(t)-v(t)\|_{n,\infty}=0
\quad\text{and}\quad
\sup_{0<t<T}\|v(t)\|_{n,\infty} \leq L_*.
\end{equation}
Moreover, since $a \in X_{\sigma}^{n,\infty}
=\overline{D(-\Delta)}^{\|\cdot\|_{n,\infty}}$,
it holds that 
$\lim\limits_{t\searrow 0}
\|e^{t\Delta}a-a\|_{n,\infty}=0$ and that
there exists $\{a_{\ep}\}_{\ep>0}\subset D(-\Delta)
\subset L_{\sigma}^{n,\infty}(\R^n)\cap L^r(\R^n)$ 
such that 
$\lim\limits_{\ep\searrow 0}\|a-a_{\ep}\|_{n,\infty}=0$, which yields $K_0(T)\to 0$ as 
$T\searrow 0$. Hence we see that $v(t)\to a$ in $L^{n,\infty}(\R^n)$ as $t\searrow 0$. See the proof of Theorem \ref{thm;maximal}.

Next put $M_{m}=M_m(T)=\sup\limits_{0<t<T}t^{\frac{1}{2}}\|\nabla v_m(t)\|_{n,\infty}$
for $m=0,1,2,\dots$. It is easy to see that
$\|\nabla e^{t\Delta}a\|_{n,\infty}\leq
Ct^{-\frac{1}{2}}\|a\|_{n,\infty}$ and 
\begin{equation*}
\begin{split}
\left\| 
\nabla \int_0^t e^{(t-s)\Delta} \PP f(s)\,ds 
\right\|_{n,\infty}
&\leq
C\int_0^t (t-s)^{-\frac{1}{2}} 
\| \PP f(s)\|_{n,\infty} \,ds\\
&\leq C\sup_{0<s<\infty}\| \PP f(s)\|_{n,\infty}t^{\frac{1}{2}},
\end{split}
\end{equation*}
for all $t>0$. Hence we have
\begin{equation*}
M_0(T)\leq \|a\|_{n,\infty} 
+ \sup_{0<t<\infty}\| \PP f(t)\|_{n,\infty} T<\infty.
\end{equation*}
Moreover, we have
\begin{equation*}
\begin{split}
\|\nabla G^*[v_m,v_m](t)\|_{n,\infty}
&\leq
C\int_0^t (t-s)^{-\frac{n}{2r}-\frac{1}{2}}
\|v_m(s)\|_r\|\nabla v_m(s)\|_{n,\infty}\,ds
\\
&\leq C_8K_mM_m t^{-\frac{1}{2}},
\end{split}
\end{equation*}
for all $0<t<T$. Hence we have 
$\sup\limits_{0<t<T}t^{\frac{1}{2}}
\|\nabla G^*[v_m.v_m](t)\|_{n,\infty}\leq 2C_8K_0 M_m$ 
for $m=0,1,2,\dots$ and
\begin{equation*}
M_{m+1} \leq M_0 + 2C_8K_0 M_m, \qquad m=0,1,2,\dots.
\end{equation*}
Then if 
\begin{equation}\label{eq;small M}
K_0 < \frac{1}{2C_8},
\end{equation}
then we have
\begin{equation*}
M_m \leq \frac{M_0}{1-2C_8K_0}=M_*,
\qquad m=0,1,2,\dots,
\end{equation*}
which yields a limit 
$\nabla v \in C\bigl((0,T)\,;\, 
L^{n,\infty}(\R^n)\bigr)$ with
\begin{equation}\label{eq;v3}
\lim_{m\to\infty} \sup_{0<t<T}
t^{\frac{1}{2}}
\|\nabla v_m(t)-\nabla v(t)\|_{n,\infty}=0
\quad\text{and}\quad
\sup_{0<t<T}t^{\frac{1}{2}}
\|\nabla v(t)\|_{n,\infty}\leq M_*.
\end{equation}
Therefore, \eqref{eq;v1}, \eqref{eq;v2} and 
\eqref{eq;v3} yield that 
\begin{equation*}
G^*[v_m,v_m](t) \to G^{*}[v,v](t)
\quad\text{in } L_{\sigma}^{n,\infty}(\R^n)
\end{equation*}
Uniformly in $t\in [0,T)$ as $m\to\infty$
Therefore, letting $m\to \infty$ 
in \eqref{eq;iteration}, we see that
\begin{equation*}
v(t)=e^{t\Delta}a 
+\int_0^t e^{(t-s)\Delta} \PP f(s)\,ds
-\int_0^t e^{(t-s)\Delta} \PP
[v\cdot\nabla v](s)\,ds, \quad 0<t<T,
\end{equation*}
which is a desired mild solution of (N-S).

Since $a\in X_{\sigma}^{n,\infty}$ and $K_0(T)\to 0$ as $T\searrow 0$ as is mentioned above,
it suffices to take small $T>0$ so that
\eqref{eq;small K}, \eqref{eq;small L} and 
\eqref{eq;small M} holds.

Furthermore, we consider the case 
$a\in X_{\sigma}^{n,\infty}\cap L^r(\R^n)$.
Put
\begin{equation*}
\eta=\min\left\{\frac{1}{4C_6},\,\frac{1}{2C_7},\,\frac{1}{2C_8}\right\}.
\end{equation*}
Then we note that
\begin{equation*}
K_0(T) \leq C\|a\|_r T^{\frac{1}{2}-\frac{n}{2r}}
+
C\sup_{0<t<\infty}\|\PP f(t)\|_{n,\infty}T.
\end{equation*}
Therefore, letting $T\to 0$, we see that there exists $T_*>0$ such that
\begin{equation*}
K_0(T_*)\leq 
C\|a\|_r T_*^{\frac{1}{2}-\frac{n}{2r}}
+
C\sup_{0<t<\infty}
\|\PP f(t)\|_{n,\infty}T_* <\eta. 
\end{equation*}
Moreover, if we put $T^*
=\sup\{T_*\leq 1\,;\, C\|a\|_r T_*^{\frac{1}{2}-\frac{n}{2r}}
+
C\sup\limits_{0<t<\infty}
\|\PP f(t)\|_{n,\infty}
T_*^{\frac{1}{2}-\frac{n}{2r}}<\eta\}$,
then the existence time $T>0$ 
is bounded from below as
\begin{equation*}
T\geq T^*=\min\left\{1, \,
\left(\frac{\eta}{C\|a\|_r
+C\sup\limits_{0<t<\infty}
\|\PP f(t)\|_{n,\infty}}\right)^{\frac{2r}{r-n}}
\right\}.
\end{equation*} 

Finally, we shall prove 
\eqref{eq;v rho} and \eqref{eq;nablav rho}.
Ler $n<\rho<\infty$. 
Then by \eqref{eq;v1} and \eqref{eq;v3} we have
\begin{equation*}
\begin{split}
\left\| \int_0^t 
e^{(t-s)\Delta}
\PP [v\cdot\nabla v](s)\,ds\right\|_{\rho}
&\leq C\int_0^t (t-s)^{-\frac{1}{2}-\frac{n}{2r}+\frac{n}{2\rho}}\|v(s)\|_r\|\nabla v(s)\|_{n,\infty}\,ds
\\
&\leq C 2K_0M_* t^{-\frac{1}{2}+\frac{n}{2\rho}},
\end{split}
\end{equation*}
for all $0<t<T$. Thus,
\begin{equation}\label{eq;Nrho}
\sup_{0<t<T}t^{\frac{1}{2}-\frac{n}{2\rho}}
\|v(t)\|_\rho \leq 
C\|a\|_{n,\infty} + C\sup_{0<t<\infty}
\|\PP f(t)\|_{n,\infty} T +2CK_0M_*=:N_{\rho}(T)<\infty.
\end{equation}
Similarly, let $n<\varrho<\rho$. 
By \eqref{eq;v3} and \eqref{eq;Nrho} we have
\begin{equation*}
\begin{split}
\left\|\int_0^t
\nabla e^{(t-s)\Delta}
\PP [v\cdot\nabla v](s)\,ds \right\|_{\varrho}
&\leq C\int_0^t 
(t-s)^{-1+\frac{n}{2\varrho}-\frac{n}{2\rho}}
\|v(s)\|_{\rho} \|\nabla v(s)\|_{n,\infty}\,ds
\\
&\leq CN_\rho(T)M_* t^{-1+\frac{n}{2\varrho}},
\end{split}
\end{equation*}
for all $0<t<T$. Hence we obtain that
\begin{equation}\label{eq;Nprho}
\sup_{0<t<T}t^{1-\frac{n}{2\varrho}} 
\| \nabla v(t)\|_{\varrho} \leq
C\|a\|_{n,\infty} 
+C \sup_{0<t<T}\| \PP f(t) \|_{n,\infty} T 
+ CN_{\rho}M_* =:N_{\varrho}^\prime(T)<\infty.
\end{equation}
Furthermore, by the Gagliardo-Nirenberg inequality
$\|v(t)\|_\infty 
\leq C\|v(t)\|_{2n}^{\frac{1}{2}}
\|\nabla v(t)\|_{2n}^{\frac{1}{2}}$,
 we obtain that
$\sup\limits_{0<t<T}t^{\frac{1}{2}}\|v(t)\|_{\infty} <\infty$.
This completes the proof of 
Theorem \ref{thm;local}.
\end{proof}
In the end of this subsection, we note that by  
\eqref{eq;v rho} and \eqref{eq;nablav rho} we note that
\begin{equation}\label{eq;nonlinear Ln}
\|v \cdot \nabla v (t)\|_{n} \leq
\|v(t)\|_{2n}\|\nabla v(t)\|_{2n}\leq
t^{-1}N_{2n}(T)N_{2n}^\prime(T)<\infty,
\qquad 0<t<T.
\end{equation} 
\subsection{Regularity of local mild solution}
In previous subsection, 
we construct a local mild solution of (N-S) 
for initial data $a\in X_\sigma^{n,\infty}$.
In this subsection, 
we discuss the mild solution
 constructed Theorem \ref{thm;local} 
 is actually a strong solution of (N-S), 
 provided the external force $f$ satisfies 
 the condition (A).

\begin{theorem}\label{thm;regularity}
Let $a \in X_\sigma^{n,\infty}$ and let
$f\in BC\bigl([0,\infty)\,;\,
 L^{n,\infty}(\R^n)\bigr)$.
 Suppose $v(t)$ is the mild solution on
 $[0,T)$ of (N-S) obtained by Theorem \ref{thm;local}.
Furthermore, let $\PP f$ be a 
H\"{o}lder continuous function on $[0,T)$ with 
value in $L_{\sigma}^{n,\infty}(\R^n)$. 
If $\PP f$ satisfies
\begin{equation*}\tag{A}
\lim_{h\searrow 0} \|e^{h\Delta} \PP f(t)
-\PP f(t)\|_{n,\infty}=0
\qquad \text{for }0\leq t <T,
\end{equation*}
The the mild solution $v(t)$ satisfies
\begin{equation*}
\frac{d}{dt}v-\Delta v +\PP [v\cdot\nabla v]=\PP f
\quad\text{in }L_{\sigma}^{n,\infty}(\R^n)
\quad \text{for }0<t<T.
\end{equation*}
\end{theorem}
Proof of Theorem \ref{thm;regularity} is rather standard. For instance, see Kozono and Ogawa \cite{Kozono Ogawa IUMJ1992} in $L^{n}$-framework,
and our previous work \cite[Theorem 5.3]{Okabe Tsutsui periodic}. However we shall give a proof for reader's convenience.
\begin{proof}
By the standard theory of abstract evolution equations, it suffices to confirm that every
$0<t_0<T$, $\PP[v\cdot\nabla v](t)$ is 
H\"{o}lder continuous on $[t_0,T)$ with value
in
$L_{\sigma}^{n,\infty}(\R^n)$. 

Firstly, we show the H\"{o}lder continuity of $e^{t\Delta}a$.
Since $L^{n,\infty}_{\sigma}(\R^n)
=\bigl(L_{\sigma}^{\frac{n}{n-1},1}(\R^n)\bigr)^*$,
the duality yields that for every $\phi \in C_{0,\sigma}^\infty(\R^n)$ it hols that
\begin{equation*}
\begin{split}
\bigl(e^{(t+h)\Delta}a-e^{t\Delta}a,\phi\bigr)
&=
\bigl(e^{\frac{t}{2}\Delta}a,
e^{(\frac{t}{2}+h)\Delta}\phi-e^{\frac{t}{2}\Delta}\phi\bigr)
\\
&=\Bigl(e^{\frac{t}{2}\Delta}a,
\int_0^1 \frac{d}{d\theta} 
e^{(\frac{t}{2}+\theta h)\Delta}\phi\,d\theta\Bigr)
\\
&= 
\Bigl(e^{\frac{t}{2}\Delta}a,
\int_0^1 
h\Delta e^{(\frac{t}{2}+\theta h)\Delta}\phi\,d\theta
\Bigr)
\\
&=h \Bigl(\Delta e^{\frac{t}{2}\Delta}a, \int_0^1 e^{(\frac{t}{2}+\theta h)\Delta}\phi\,d\theta\Bigr)
\end{split}
\end{equation*}
for all $t\in [t_0,T)$ and all $0<h<T-t$.
Therefore it holds that
\begin{equation*}
\bigl|\bigl(e^{(t+h)\Delta}a-e^{t\Delta}a,
\phi\bigr)\bigr|\leq 
Ch \|\Delta e^{\frac{t}{2}\Delta}a\|_{n,\infty} 
\|\phi\|_{\frac{n}{n-1},1}
\leq Ch \sup_{t_0<t<T}\|\Delta e^{\frac{t}{2}\Delta}a\|_{n,\infty}\|\phi \|_{\frac{n}{n-1},1},
\end{equation*}
which implies $\|e^{(t+h)\Delta}a-e^{t\Delta}a\|_{n,\infty}\leq
Ch \sup\limits_{t_0<t<T}\|\Delta e^{\frac{t}{2}\Delta}\phi\|_{n,\infty}$.
Similarly, we have
\begin{equation*}
\begin{split}
\|\nabla e^{(t+h)\Delta}a-
\nabla e^{t \Delta}a\|_{n,\infty}
&\leq C\Bigl(\frac{t}{2}\Bigr)^{-\frac{1}{2}}
\| e^{(\frac{t}{2}+h)\Delta}a 
- e^{\frac{t}{2}\Delta}a \|_{n,\infty}
\\
&\leq C\Bigl(\frac{t_0}{2}\Bigr)^{-\frac{1}{2}}
h \sup_{t_0<t<T}\|\Delta e^{\frac{t}{4}\Delta}a\|_{n,\infty}
\end{split}
\end{equation*}
for all $t\in [t_0,T)$ and all $0<h<T-t$.

Next we recall $\PP f$ is H\"{o}lder continuous on
$[0,T]$ in $L^{n,\infty}(\R^n)$ with some order $\alpha>0$, i.e.,
\begin{equation*}
 \|\PP f(t)-\PP f(s)\|_{n,\infty} \leq C_T
 |t-s|^\alpha
 \qquad\text{for }0\leq t, s \leq T,
\end{equation*}
where $C_T>0$ is independent of $t$ and $s$.
By changing variable $t=s-h$, we see that
\begin{multline*}
\int_0^{t+h} e^{(t+h-s)\Delta}\PP f(s)\,ds
-\int_0^t e^{(t-s)\Delta}\PP f(s)\,ds \\
=\int_0^h e^{(t+h-s)\Delta}\PP f(s)\,ds
+\int_h^{t+h} e^{(t+h-s)\Delta}\PP f(s)\,ds
-\int_0^t e^{(t-s)\Delta}\PP f(s)\,ds \\
=
\int_0^h e^{(t+h-s)\Delta}\PP f(s)\,ds
+\int_0^t e^{(t-s)\Delta}
[\PP f(s+h)-\PP f(s)]\,ds=:I_1+I_2.
\end{multline*}
We estimate $I_1$ and $\nabla I_1$ as follow,
\begin{equation*}
\|I_1\|_{n,\infty} \leq \int_0^t 
\|e^{(t+h-s)\Delta}\PP f(s)\|_{n,\infty}\,ds
\leq Ch \sup_{0<s<T} \|\PP f(s)\|_{n,\infty}
\end{equation*}
and 
\begin{equation*}
\|\nabla I_1\|_{n,\infty}
\leq
\int_0^h \|\nabla e^{(h-s)\Delta}e^{t\Delta}\PP f(s)\|_{n,\infty}\,ds
\leq
C\int_0^h (h-s)^{-\frac{1}{2}} 
\|\PP f(s)\|_{n,\infty}\,ds 
\leq Ch^{\frac{1}{2}}
\sup_{0<s<T}\|\PP f(s) \|_{n,\infty}
\end{equation*}
for all $0\leq t<T$.
Next we estimate $I_2$ and $\nabla I_2$,
\begin{equation*}
\|I_2\|_{n,\infty} \leq 
\int_0^t \|e^{(t-s)\Delta} 
[\PP f(s+h)-\PP f(s)]\|_{n,\infty}\,ds
\leq Ch^{\alpha} T
\end{equation*}
and
\begin{equation*}
\|\nabla I_2\|_{n,\infty}
\leq
C\int_0^t (t-s)^{-\frac{1}{2}}
\|\PP f(s+h)-\PP f(s)\|_{n,\infty}\,ds
\leq C h^\alpha T^{\frac{1}{2}}
\end{equation*}
for all $0\leq t< T$.

In order to derive that $G^*[v,v](\cdot)$ is 
H\"{o}lder continuous on $[t_0,T)$ in 
$L^{n,\infty}_{\sigma}(\R^n)$, we write
\begin{equation*}
\begin{split}
G^*[v,v](t+h)-G^{*}[v,v](t)
&=
-\int_t^{t+h} e^{(t+h-s)\Delta} \PP [v\cdot\nabla v](s)\,ds
-
\int_0^t [e^{(t+h-s)\Delta}-e^{(t-s)\Delta}]
\PP [v\cdot\nabla v](s)\,ds
\\
&=: J_1+ J_2.
\end{split}
\end{equation*}
Since $\PP[v\cdot\nabla v] \in 
BC\bigl([t_0,T)\,;\,
L^{n}_{\sigma}(\R^n)\bigr)$ 
by \eqref{eq;nonlinear Ln}, 
we have
\begin{equation*}
\| J_1\|_{n,\infty}
\leq
Ch \sup_{t_0<s<T}\|\PP [v\cdot \nabla v](s)\|_n
\end{equation*}
and
\begin{equation*}
\| \nabla J_1\|_{n,\infty} 
\leq
C\int_{t}^{t+h} (t+h-s)^{-\frac{1}{2}}
\| \PP [v\cdot\nabla v](s)\|_{n}\,ds
\leq Ch^{\frac{1}{2}}
\sup_{t_0<s<T}\|\PP [v\cdot\nabla v](s)\|_n
\end{equation*}
for all $t\in (t_0,T)$ and for all $0<h<T-t$.
Because,
\begin{equation*}
J_2 =
\lim_{\ep\searrow 0} J_{2,\ep}
:=\lim_{\ep\searrow 0}
\int_{\ep}^{t-\ep}
[e^{(t+h-s)\Delta}-e^{(t-s)\Delta}]
\PP [v\cdot\nabla v](s)\,ds,
\end{equation*}
it is enough to bound $J_{2,\ep}$
for sufficiently small $\ep>0$,
instead of $J_2$.
Thus we observe, 
\begin{equation*}
J_{2,\ep} = \int_{\ep}^{t-\ep}
\int_0^1 \frac{d}{d\theta} e^{(t-s+\theta h)\Delta} \PP [v\cdot\nabla v](s)\,d\theta ds
=
h \int_{\ep}^{t-\ep}\int_0^1 \Delta e^{(t-s+\theta h)\Delta} \PP [v\cdot\nabla v](s)\,d\theta ds
\end{equation*}
for all $t \in [t_0,T)$ and all $0<h<T-t$.
By \eqref{eq;v1} and \eqref{eq;v3} we have
\begin{equation*}
\begin{split}
\| J_{2,\ep}\|_{n} 
&\leq
h \int_{\ep}^{t-\ep}
\int_0^1 \|(-\Delta)^{\frac{1}{2}}
e^{\theta h\Delta}(-\Delta)^{\frac{1}{2}}
e^{(t-s)\Delta}\PP [v\cdot\nabla v](s)\|_{n}
\,d\theta ds
\\
&\leq
Ch \int_{\ep}^{t-\ep}\int_0^1
(\theta h)^{-\frac{1}{2}}\,d\theta
(t-s)^{-\frac{n}{2r}-\frac{1}{2}}
\|v(s)\|_r \|\nabla v(s)\|_{n,\infty}\, ds
\\
&\leq
C h^{\frac{1}{2}} \int_0^t
(t-s)^{-\frac{n}{2r}-\frac{1}{2}} 
s^{-1+\frac{n}{2r}}\,ds
\sup_{0<s<T} s^{\frac{1}{2}-\frac{n}{2r}}
\|v(s)\|_r 
\sup_{0<s<T} s^{\frac{1}{2}}\|\nabla v(s) \|_{n,\infty}\,ds
\\
&\leq
Ch^{\frac{1}{2}} K_0M_* t^{-\frac{1}{2}}
\\
&\leq
Ch^{\frac{1}{2}} K_0M_* t_0^{-\frac{1}{2}},
\end{split}
\end{equation*}
and by \eqref{eq;v3} and \eqref{eq;Nrho} with 
$2n<\rho$, i.e., $\frac{n}{2\rho}<\frac{1}{4}$,
\begin{equation*}
\begin{split}
\| \nabla J_{2,\ep}\|_n
&\leq
Ch \int_{\ep}^{t-\ep} \int_0^1
\| (-\Delta)^{\frac{3}{4}} e^{\theta h\Delta}
(-\Delta)^{\frac{3}{4}} e^{(t-s)\Delta}
\PP [v\cdot \nabla v](s)\|_n\, d\theta ds
\\
&\leq
Ch^{\frac{1}{4}} \int_{\ep}^{t-\ep}
(t-s)^{-\frac{n}{2\rho}-\frac{3}{4}} 
\|v(s)\|_\rho \|\nabla v(s)\|_{n,\infty}\,ds
\\
&\leq 
Ch^{\frac{1}{4}}\int_0^t 
(t-s)^{-\frac{n}{2\rho}-\frac{3}{4}}
s^{-1+\frac{n}{2\rho}}\,ds
\sup_{0<s<T} s^{\frac{1}{2}-\frac{n}{2\rho}}
\|v(s)\|_\rho 
\sup_{0<s<T} s^{\frac{1}{2}}
\|\nabla v(s)\|_{n,\infty}
\\
&\leq
Ch^{\frac{1}{4}} N_\rho (T) M_* t^{-\frac{3}{4}}
\\
&\leq
Ch^{\frac{1}{4}} N_\rho (T) M_* t_0^{-\frac{3}{4}}, 
\end{split}
\end{equation*}
for all $t\in [t_0,T)$, $0<h < T-t$ and for all $0<\ep<t/2$.
Therefore, there exists $\beta>0$ so that
\begin{equation}\label{eq;holder nonlinear}
\| v(t+h)-v(t)\|_{n,\infty}
\leq C_T h^\beta
\quad\text{and}\quad
\| \nabla v(t+h)-v(t)\|_{n,\infty}
\leq C_T h^\beta
\end{equation}
for all $[t_0,T)$ and $0<h<T-t$, where the
constant $C_T>0$ is independent of $t$ and $h$.

Hence \eqref{eq;holder nonlinear} yields that
$\PP [v\cdot \nabla v](\cdot)$ is 
H\"{o}lder continuous on $[t_0,T)$ in $L^{n,\infty}_{\sigma}(\R^n)$.
Indeed, we can see 
\begin{multline*}
\|\PP[v\cdot\nabla v](t+h)-\PP[v\cdot\nabla v](t)
\|_{\frac{n}{2},\infty}
\\
\leq
C\|\bigl(v(t+h)-v(t)\bigr)\cdot \nabla v(t+h)
\|_{\frac{n}{2},\infty}
+C\| v(t)\cdot
\bigl(\nabla v(t+h)-\nabla v(t)\bigr)
\|_{\frac{n}{2},\infty}
\\
\leq C_Th^{\beta}
\left(\sup_{t_0<t<T}\|\nabla v(t)\|_{n,\infty}
+\sup_{t_0<t<T} \|v(t)\|_{n,\infty}
\right)
\end{multline*}
for all $t\in[t_0,T)$ and all $0<h<T-t$.
On the other hand,
$\sup\limits_{t_0<t<T}\|\PP[v\cdot\nabla v](t)\|_{r}<\infty$ for some $r>n$ with \eqref{eq;Nrho}
and \eqref{eq;Nprho} for suitable $\rho>n$. 
The H\"{o}lder continuity of $\PP[v\cdot\nabla v]$
in $L^{n,\infty}_{\sigma}(\R^n)$ can be derived by
those inequalities above 
via a H\"{o}lder interpolation inequality.

Finally, we note that $\PP [v\cdot \nabla v](\cdot)$ also satisfies the assumption (A) since
$\PP [v\cdot \nabla v](t)\in L_{\sigma}^{n}(\R^n)$ for $t>0$.
By the virtue of Lemma \ref{thm;abstract evo} 
our mild solution $v$ obtained by Theorem \ref{thm;local}
is actually the strong solution on $[t_0,T)$ of
(N-S), i.e., it holds that 
\begin{equation*}
\frac{d}{dt}v(t) -\Delta v(t) 
+\PP[v\cdot \nabla v](t)=\PP f(t)
\quad\text{in }
L_{\sigma}^{n,\infty}(\R^n)
\quad \text{for }t_0<t<T,
\end{equation*}
whenever $\PP f$ satisfies the assumption (A).
Since $0<t_0<T$ is arbitrary, the strong solution
$v$ satisfies the differential equation on
$0<t<T$. This completes the proof.
\end{proof} 

\subsection{Uniqueness of local mild solutions}
The uniqueness of the mild solution of (N-S) 
is already established, see Kozono and Yamazaki 
\cite{Kozono Yamazaki MZ1998}. 
For just reader's convenience 
we give the proof of the following theorem.
\begin{theorem}\label{thm;local uniqueness}
Let $n<r<\infty$. Then there exists a constant 
$\kappa=\kappa(n,r)>0$ 
with the following property.
Let $a\in X_{\sigma}^{n,\infty}$ or $a \in L^{n,\infty}_{\sigma}(\R^n)\cap L^r(\R^n)$, and let
$\PP f \in BC([0,T)\,;\, L^{n,\infty}_{\sigma}(\R^n))$. 
Suppose $v$ is the mild solution on $(0,T)$ of (N-S) obtained by Theorem \ref{thm;local}.
Suppose $w$ is also a mild solution on $(0,T)$
of (N-S) which satisfies 
$t^{\frac{1}{2}-\frac{n}{2r}}
w \in BC\bigl((0,T)\,;\,L^{r}(\R^n)\bigr)$.
If 
\begin{equation}\label{eq;uniqueness}
\limsup_{t\to0} t^{\frac{1}{2}-\frac{n}{2r}}
\|w(t)\|_r\leq \kappa
\end{equation}
then $v\equiv w$ on $[0,T)$.
\end{theorem}
\begin{remark}\rm
If $a\in X_{\sigma}^{n,\infty}$, we are able to omit \eqref{eq;uniqueness}, since 
$t^{\frac{1}{2}-\frac{n}{2r_*}}\|G^*[w,w](t)\|_{r_*}\to 0$ as $t\to 0$ for some $n<r_*<r$
and since $t^{\frac{1}{2}-\frac{n}{2r_*}}\|e^{t\Delta}a\|_{r_*}\to 0$ as $t\to 0$.
We assume \eqref{eq;uniqueness} only for 
the simplicity for the proof.

Furthermore, 
the argument introduced by Brezis \cite{Brezis}
enables us to remove the condition 
$t^{\frac{1}{2}-\frac{n}{2r}}
w \in BC\bigl((0,T)\,;\,L^{r}(\R^n)\bigr)$.
Namely, the mild solution $w \in BC\bigl([0,T)\,;\,L^{n,\infty}_{\sigma}(\R^n)\bigr)\cap
C\bigl((0,T)\,;\,L^r(\R^n)\bigr)$ with some $r>n$ of (N-S) necessarily 
satisfies \eqref{eq;uniqueness}
with $\kappa=0$. See appendix.
\end{remark}
\begin{proof}
By \eqref{eq;v1} and 
$a\in L^{n,\infty}_{\sigma}(\R^n)\cap 
L^r(\R^n)$ we note that
\begin{equation}
\sup_{0<s<t} s^{\frac{1}{2}-\frac{n}{2r}}\|v(s)\|_{r}
\leq 2K_0(t)\to 0 \quad \text{as }t \to 0.
\end{equation}
Define $U(t)=v(t)-w(t)$ and
$D(\tau,t)=\sup\limits_{\tau<s<t}\|U(s)\|_{n,\infty}$,
then
\begin{equation*}
U(t)=G[U,v](t)+G[w,U](t),\qquad 0<t<T.
\end{equation*} 
For every $\phi \in C_{0,\sigma}^\infty(\R^n)$, we have
\begin{equation*}
\begin{split}
\bigl|\bigl(G[U,v](t),\phi\bigr)\bigr|
&\leq
\int_0^t \bigl|\bigl(U(s)\cdot\nabla e^{-(t-s)\Delta}\phi,v(s)
\bigr)\bigr|\,ds
\\
&\leq
C\int_0^t \|U(s)\|_{n,\infty}\|v(s)\|_r
(t-s)^{-\frac{n}{2r}-\frac{1}{2}}
\|\phi\|_{\frac{n}{n-1},1}\,ds
\\
&\leq
C_7\sup_{0<s<t}s^{\frac{1}{2}-\frac{n}{2r}}\|v(s)\|_r\,D(0,t)
\|\phi\|_{\frac{n}{n-1},1}.
\end{split}
\end{equation*}
Therefore one obtains 
\begin{equation*}\|G[U,v](t)\|_{n,\infty}\leq
2C_7 K_0(t)D(0,t).
\end{equation*} 
Similarly,
\begin{equation*}\|G[w,U](t)\|_{n,\infty}
\leq C_7 \sup\limits_{0<s<t}s^{\frac{1}{2}-\frac{n}{2r}}\|w(s)\|_rD(0,t).
\end{equation*}
Now, let us define the constant 
$\kappa=\kappa(n,r)$ of \eqref{eq;uniqueness} as
\begin{equation*}
\kappa=\frac{1}{4C_7}.
\end{equation*}
Then by the assumption there exists $0<t_0\leq T$ such that
\begin{equation*}
2C_7K_0(t_0)+C_7 \sup\limits_{0<s<t_0}s^{\frac{1}{2}-\frac{n}{2r}}\|w(s)\|_r
\leq \frac{1}{2}.
\end{equation*}
Therefore one obtains
\begin{equation*}
D(0,t_0) \leq \bigl(2C_7K_0(t_0)+C_7 \sup\limits_{0<s<t_0}s^{\frac{1}{2}-\frac{n}{2r}}\|w(s)\|_r\bigr)D(0,t_0)
\leq \frac{1}{2}D(0,t_0).
\end{equation*}
Since $D(0,t)$ is a non-decreasing function in $t$, we conclude that
$U(t)\equiv 0$ for $0<t\leq t_0.$

Next we assume $U\equiv 0$ on $(0,\tau]$ for some $0<\tau< T$.
For every $\phi \in C_{0,\sigma}^\infty(\R^n)$,
\begin{equation*}
\begin{split}
|(U(t),\phi)|
&\leq
\bigl|\bigl(G[U,v](t),\phi\bigr)\bigr|
+
\bigl|\bigl(G[w,U](t),\phi\bigr)\bigr|
\\
&\leq
\int_\tau^t \bigl|\bigl(U(s)\cdot\nabla e^{-(t-s)\Delta}\phi,v(s)
\bigr)\bigr|\,ds
+
\int_\tau^t \bigl|\bigl(w(s)\cdot\nabla e^{-(t-s)\Delta}\phi,U(s)
\bigr)\bigr|\,ds
\\
&\leq
C\int_\tau^t 
\bigl(\|v(s)\|_r+\|w(s)\|_r\bigr)
\|U(s)\|_{n,\infty}
(t-s)^{-\frac{n}{2r}-\frac{1}{2}}
\|\phi\|_{\frac{n}{n-1},1}\,ds
\\
&\leq
C\Bigl(
\sup_{t_0<s<T}\|v(s)\|_r+\sup_{t_0<s<T}\|w(s)\|_r
\Bigr)
D(\tau,t)
(t-\tau)^{\frac{1}{2}-\frac{n}{2r}}
\|\phi\|_{\frac{n}{n-1},1}.
\end{split}
\end{equation*}
Hence by the duality, it holds
\begin{equation}
D(\tau,t)\leq C_9\Bigl(
\sup_{t_0<s<T}\|v(s)\|_r+\sup_{t_0<s<T}\|w(s)\|_r\Bigr)D(\tau,t)
(t-\tau)^{\frac{r-n}{2r}},
\qquad \tau<t<T.
\end{equation}
Here, we put $\tau_0>0$ as
\begin{equation*}
\tau_0=\min\left\{T,\,\left(
\frac{1}{2C_9\bigl(\sup\limits_{t_0<s<T}\|v(s)\|_r+\sup\limits_{t_0<s<T}\|w(s)\|_r\bigr)}\right)^{\frac{2r}{r-n}}\right\}.
\end{equation*}
Then one sees that
\begin{equation*}
D(\tau,\tau+\tau_0) \leq \frac{1}{2}D(\tau,\tau+\tau_0).
\end{equation*}
This implies that $U\equiv 0$ on $(0,\tau+\tau_0)$.
Consequently, $U\equiv 0$ on $(0,T)$. This completes the proof of 
Theorem \ref{thm;local uniqueness}.
\end{proof}

\section{Completion of 
the proof of Theorem \ref{thm;global}}
\label{sec;p2.1}

In this section, we prove (iv) and (vii) 
of Theorem \ref{thm;global} 
with the aid of Theorem \ref{thm;local}, 
Theorem \ref{thm;regularity} 
and Theorem \ref{thm;local uniqueness}.

Recall that the mild solution $u$ of (N-S) 
obtained by (iii) and (vi) 
in Theorem \ref{thm;global}.

Since $a \in L^{n,\infty}_{\sigma}(\R^n)\cap
L^r(\R^n)$ 
with $\frac{1}{r}=\frac{1}{p}-\frac{2}{n}$ for some $\frac{n}{3}<p<\frac{n}{2}$, we are able to
construct 
a local in time mild soltion $v$ of (N-S)
on $(0,T_*)$ by Theorem \ref{thm;local} and 
Remark \ref{rem;local}.
Since $\PP f$ is H\"{o}lder continuous on $[0,\infty)$ in $L^{n,\infty}_{\sigma}(\R^n)$
and satisfies the condition (A),
by Theorem \ref{thm;regularity} we see that
 $v$ is a 
strong solution of (N-S).
Since $u \in BC\bigl((0,\infty)\,;\,
L^r(\R^n)\bigr)$, 
Theorem \ref{thm;local uniqueness} 
guarantees $u\equiv v$ on $(0,T_*)$.

Next, put
\begin{equation*}
K^*:= \sup_{T_*/2<t<\infty}\|u(t)\|_r <\infty.
\end{equation*}
Then for $T_*/2<t_0<\infty$, we construct a local
mild solution $v$ of (N-S) 
with $v|_{t=0}=u(t_0)$ on $(t_0,t_0+\tau_0)$, where
\begin{equation*}
\tau_0=\min\left\{
1, \Bigl(\frac{\eta}{K^*
+\sup\limits_{0<t<\infty}
\|\PP f(t)\|_{n,\infty}}\Bigr)^{\frac{2r}{r-n}}
\right\}.
\end{equation*}
Sine $v$ is actually a strong solution of (N-S) and since $\tau_0$ is independent of $t_0$,
the uniqueness theorem yields that
$u$ is a strong solution of (N-S) on $(0,\infty)$.
This completes all of 
the proof of Theorem \ref{thm;global}.

\section{Completion of the proof of Theorem \ref{thm;localX}}\label{sec;p2.3}

We prove (iii) of Theorem \ref{thm;localX}.
We assume $a\in X_{\sigma}^{n,\infty}$ and $f\in BC\bigl([0,\infty)\,;\,
\widetilde{L}^{\frac{n}{3},\infty}(\R^n)\bigr)$
for $n\geq 4$ and $f\in BC\bigl([0,\infty)\,;\,
L^1(\R^n)\bigr)$ 
and assume $\PP f$ is H\"{o}lder
continous on $[0,\infty)$ in 
$L^{n,\infty}(\R^n)$ with $\PP f(t)\in X_{\sigma}^{n,\infty}$ for $t\geq 0$.
Let $u\in BC\bigl([0,T)\;\,
X_{\sigma}^{n,\infty}\bigr)$ be a weak mild solution of (N-S) obtained by (i) or (ii). 
On the other hand, by Theorem \ref{thm;local}
and Theorem \ref{thm;regularity} for $a$ and $f$
we have a strong solution
$w\in BC\bigl([0,\tau)\,;\,
L_{\sigma}^{n,\infty}(\R^n)\bigr)$ of (N-S) with some $\tau>0$ which is actually a weak mild solution of (N-S).
Then the argument in subsection \ref{subsec;UWMS}
is applicable to $u$ and $w$. Hence, we see that $u\equiv w$ on $[0,T)$.
This completes the proof. 

\section{Proof of Theorem \ref{thm;brezis}}
Suppose $a \in L_{\sigma}^{n,\infty}(\R^n)$
and $f\in BC\bigl([0,T)\,;\,
\widetilde{L}^{\frac{n}{3},\infty}(\R^n)\bigr)$,
$n\geq 4$ and 
$f \in BC\bigl( [0,T)\,;\,L^1(\R^3)\bigr)$.
Let $v \in BC\bigl([0,T)\,;\,
\widetilde{L}_{\sigma}^{n,\infty}(\R^n)\bigr)$ 
be a weak mild solution of (N-S).
By Lemma \ref{lem;tildeL} we see that
$F(t)\to0$  and $G[v,v](t)\to 0$ 
in $L^{n,\infty}(\R^n)$ as $t\to 0$, where 
$F(t)=\int_0^t \PP e^{(t-s)\Delta}f(s)\,ds$ and
$G[u,v](t)=-\int_0^t \nabla\cdot e^{(t-s)\Delta}\PP [u\otimes v](s)\,ds$. Then we see that 
\begin{equation*}
e^{t\Delta}a -a=
u(t)-a -F(t)-
G[u,u](t)\to 0 \quad
\text{in }L_{\sigma}^{n,\infty}(\R^n)\qquad\text{as }t\to0.
\end{equation*}
Hence, we conclude $a \in X_{\sigma}^{n,\infty}$.
Therefore Theorem \ref{thm;local weak uniqueness}
implies the proof of (i) 
of Theorem \ref{thm;brezis}.

Next we consider 
$a \in \widetilde{L}_{\sigma}^{n,\infty}(\R^n)$
and $u\in BC\bigl([0,T)\,;\,
L_{\sigma}^{n,\infty}(\R^n)\bigr)$.
Then by Lemma \ref{lem;tildeL} we have
\begin{equation*}
G[u,u](t)=G[u-a,u](t)+G[a,u](t)\to 0
\quad \text{in }L^{n,\infty}(\R^n)
\quad \text{as } t\to 0.
\end{equation*}
Hence by the same argument as above, 
we obtain $a \in X_\sigma^{n,\infty}$.
This completes the proof of (ii) of 
Theorem \ref{thm;brezis}.

\appendix
\section{Uniqueness criterion}
In this section, we show 
$\lim\limits_{t\to0} t^{\frac{1}{2}-\frac{n}{2r}}\|u(t)\|_{r}=0$ for a (weak) mild solution
\begin{equation}\label{eq;class Brezis}
u \in BC\bigl([0,T)\,;\, 
L_{\sigma}^{n,\infty}(\R^n)\bigr)
\cap
C\bigl((0,T)\,;\, L^r(\R^n)\bigr)
\quad\text{for some }r>n,
\end{equation}
by the argument introduced by Brezis \cite{Brezis}. 
\begin{lemma}\label{lem;Brezis}
Let $a\in L_{\sigma}^{n,\infty}(\R^n)$
and $u$ be a (weak) mild solution on
$(0,T)$ in the class
\eqref{eq;class Brezis} for some $r>n$.
Suppose $\Gamma_u=\{u(t)\in 
L_{\sigma}^{n,\infty}(\R^n)
\cap L^r(\R^n)\,;\,t\in(0,T^\prime]\}$ for some $0<T^\prime<T$.
Then there exists a monotonously nondecreasing  function $\delta(t;\Gamma_u)$ with $\delta(0;\Gamma_u)=0$ on $[0,T^\prime]$ such that
\begin{equation*}
t^{\frac{1}{2}-\frac{n}{2r}}
\|e^{t \Delta} f\|_r 
\leq \delta (t;\Gamma_u)
\qquad\text{for all } f \in \overline{\Gamma_u}^{\|\cdot\|_{n,\infty}}. 
\end{equation*}
\end{lemma}
\begin{proof}
Observe that $\Gamma_u$ is precompact in $L^{n,\infty}(\R^n)$ since $u$ is continuous on $[0,T^\prime]$ in $L^{n,\infty}_{\sigma}(\R^n)$.
Then $
\delta(t;f)=\sup\limits_{0<s\leq t}s^{\frac{1}{2}-\frac{n}{2r}}\|e^{s\Delta}f\|_{r}$ is continuous and monotonously nondecreasing function on $[0,T^\prime]$ 
with 
$\lim\limits_{t\to0}\delta(t,f)=0,
$
for every fixed $f \in 
\overline{\Gamma_u}^{\|\cdot\|_{n,\infty}}$.
On the other hand, for every fixed 
$0<t\leq T^\prime$, $\delta(t,f)$ is continuous on $\overline{\Gamma_u}^{\|\cdot\|_{n,\infty}}$. So 
\begin{equation*}
\delta(t;\Gamma_u):= \max_{f\in \overline{\Gamma_u}} \delta(t,f)
\end{equation*}
is a desired function. This completes the proof.
\end{proof}

Next, we shall prove $\lim\limits_{t\to0} t^{\frac{1}{2}-\frac{n}{2r}}\|u(t)\|_{r}=0$ with above lemma.
For a while we fix $0<t_0<T^\prime$.
Note that
\begin{equation*}
u(t+t_0)=
e^{t\Delta}u(t_0)+
\int_{0}^t e^{(t-s)\Delta}\PP f(s+t_0)\,ds
-\int_0^t e^{(t-s)\Delta}
\PP[u\cdot\nabla u](s+t_0)\,ds,
\end{equation*}
for all $0<t<T^\prime-t_0$.
Put $\gamma(t):=\sup\limits_{0<s\leq t}s^{\frac{1}{2}-\frac{n}{2r}}\|u(s+t_0)\|_r$. Then we see that
\begin{equation*}
\gamma(t)\leq \delta(t;\Gamma_u)+
Ct \sup_{0<s<T^\prime}\|\PP f(s)\|_{n,\infty}
+C_9 \gamma(t)^2.
\end{equation*}
Now we choose $t^*>0$ so that
\begin{equation*}
\delta(t^*;\Gamma_u)+
Ct^{*}\sup_{0<s<T^\prime}\|\PP f(s)\|_{n,\infty}
<\frac{1}{4C_{10}}.
\end{equation*}
On the other hand, 
since $\gamma(t)$ is continuous 
and $\lim\limits_{t\to 0}\gamma(t)=0$,we put 
$t_*:=\sup\{t\,;\, \gamma(t)<1/C_{10} \text{ on } [0,t]\}$. 
It holds that
\begin{equation*}
\gamma(t)\leq 2\bigl(
\delta(t;\Gamma_u)+
Ct\sup_{0<s<T^\prime}\|\PP f(s)\|_{n,\infty}
\bigr)
\leq \frac{1}{2C_{10}}
\qquad\text{for  } t<\min\{t_*,t^*\}.
\end{equation*}
Hence this yields that $t^*\leq t_*$ and
\begin{equation*}
\|u(t+t_0)\|_{r} \leq 
\bigl(\delta(t;\Gamma_u)+
Ct\sup_{0<s<T^\prime}\|\PP f(s)\|_{n,\infty}
\bigr)t^{-\frac{1}{2}+\frac{n}{2r}},
\qquad t\in (0,t^*).
\end{equation*}
Since $t^{*}$ is independent of $t_0$, letting $t_0\to 0$ we obtain that
\begin{equation*}
\|u(t)\|_{r} \leq 
\bigl(\delta(t;\Gamma_u)+
Ct\sup_{0<s<T^\prime}\|\PP f(s)\|_{n,\infty}
\bigr)t^{-\frac{1}{2}+\frac{n}{2r}},
\qquad t\in (0,t^*).
\end{equation*} 
Therefore, $\lim\limits_{t\to0}t^{\frac{1}{2}-\frac{n}{2r}}\|u(t)\|_r=0$.
This completes the proof.

\medskip
\noindent\textbf{Acknoledgement}
The authors would like to thank Professor Yasushi Taniuchi for fruitful comments.
They also would like to appreciate Dr. Naoto Kajiwara for his comment.
The work of the first author is partly supported by JSPS Grand-in-Aid for Young 
Scientists (B) 17K14215.
The work of the second author is partly supported by JSPS through Grand-in-Aid for Young Scientists (B) 15K20919.

\end{document}